\documentclass[12pt]{amsart}
\usepackage{amscd,amsmath,amsthm,amssymb,graphics}

\usepackage{amsmath}
\usepackage{amssymb}
\usepackage{amsbsy}
\usepackage{graphicx}
\usepackage{amsfonts}

\usepackage{amscd,amsmath,amsthm,amssymb}
\usepackage{pstcol,pst-plot,pst-3d}
\usepackage{lmodern,pst-node}

\usepackage{pgf,tikz}
\usetikzlibrary{arrows}

\newcommand{\bm}[1]{\mbox{\boldmath{$#1$}}}

\def\NZQ{\Bbb}               

\def\ZZ{{\NZQ Z}}
\def\RR{{\NZQ R}}

\def\frk{\frak}               

\def\Phi{{\frk n}}
\def\Phi{{\frk N}}
\def\MP{{\mathcal P}}

\def\MM{{\mathcal M}}

\def\ab{{\bold a}}
\def\bb{{\bold b}}
\def\xb{{\bold x}}

\def\tb{{\bold t}}
\def\wb{{\bold w}}
\def\ub{{\bold u}}
\def\Pb{{\bold P}}
\def\vb{{\bold v}}
\def\opn#1#2{\def#1{\operatorname{#2}}} 
\opn\chara{char} \opn\length{\ell} \opn\pd{pd} \opn\rk{rk}
\opn\projdim{proj\,dim} \opn\injdim{inj\,dim} \opn\rank{rank}
\opn\depth{depth} \opn\grade{grade} \opn\height{height}
\opn\embdim{emb\,dim} \opn\codim{codim}

\opn\Tr{Tr} \opn\bigrank{big\,rank}
\opn\superheight{superheight}\opn\lcm{lcm}
\opn\trdeg{tr\,deg}
\opn\reg{reg} \opn\lreg{lreg} \opn\ini{in} \opn\lpd{lpd}
\opn\size{size}\opn\bigsize{bigsize}
\opn\cosize{cosize}\opn\bigcosize{bigcosize}
\opn\sdepth{sdepth}\opn\sreg{sreg}
\opn\link{link}\opn\fdepth{fdepth}
\opn\index{index}
\opn\index{index}
\opn\indeg{indeg}
\opn\N{N}
\opn\SSC{SSC}
\opn\SC{SC}
\opn\conv{conv}
\opn\div{div} \opn\Div{Div} \opn\cl{cl} \opn\Cl{Cl}
\opn\Spec{Spec} \opn\Supp{Supp} \opn\supp{supp} \opn\Sing{Sing}
\opn\Ass{Ass} \opn\Min{Min}\opn\Mon{Mon} \opn\dstab{dstab} \opn\astab{astab}
\opn\Syz{Syz}
\opn\reg{reg}
\opn\Ann{Ann} \opn\Rad{Rad} \opn\Soc{Soc}
\opn\Im{Im} \opn\Ker{Ker} \opn\Coker{Coker} \opn\Am{Am}
\opn\Hom{Hom} \opn\Tor{Tor} \opn\Ext{Ext} \opn\End{End}
\opn\Aut{Aut} \opn\id{id}

\opn\nat{nat}
\opn\pff{pf}
\opn\Pf{Pf} \opn\GL{GL} \opn\SL{SL} \opn\mod{mod} \opn\ord{ord}
\opn\Gin{Gin} \opn\Hilb{Hilb}\opn\sort{sort}
\opn\initial{init}
\opn\ende{end}
\opn\height{height}
\opn\type{type}
\opn\aff{aff} \opn\con{conv} \opn\relint{relint} \opn\st{st}
\opn\lk{lk} \opn\cn{cn} \opn\core{core} \opn\vol{vol}
\opn\link{link} \opn\star{star}\opn\lex{lex}\opn\Mon{Mon}\opn\Min{Min}
\opn\gr{gr}

\def\pot#1#2{#1[\kern-0.28ex[#2]\kern-0.28ex]}

\opn\dirlim{\underrightarrow{\lim}}
\opn\inivlim{\underleftarrow{\lim}}
\let\union=\cup
\let\sect=\cap

\let\to=\rightarrow

\def\Implies{\ifmmode\Longrightarrow \else
        \unskip${}\Longrightarrow{}$\ignorespaces\fi}
\def\implies{\ifmmode\Rightarrow \else
        \unskip${}\Rightarrow{}$\ignorespaces\fi}
\def\iff{\ifmmode\Longleftrightarrow \else
        \unskip${}\Longleftrightarrow{}$\ignorespaces\fi}

\let\:=\colon
\newtheorem{Theorem}{Theorem}[section]
 \newtheorem{Lemma}[Theorem]{Lemma}
 \newtheorem{Corollary}[Theorem]{Corollary}
 \newtheorem{Proposition}[Theorem]{Proposition}
 \newtheorem{Remark}[Theorem]{Remark}
 
 \newtheorem{Example}[Theorem]{Example}
 
 \newtheorem{Definition}[Theorem]{Definition}
 \newtheorem*{Definition*}{Definition}
 
 \newtheorem{Conjecture}{Conjecture}
 \newtheorem*{Conjecture*}{Conjecture}

\newtheorem{Running Example}[Theorem]{Running Example}

 \newtheorem{Discussion}[Theorem]{Discussion}

\let\epsilon\varepsilon
\let\kappa=\varkappa
\def\qed{\ifhmode\textqed\fi
      \ifmmode\ifinner\quad\qedsymbol\else\dispqed\fi\fi}
\def\textqed{\unskip\nobreak\penalty50
       \hskip2em\hbox{}\nobreak\hfil\qedsymbol
       \parfillskip=0pt \finalhyphendemerits=0}
\def\dispqed{\rlap{\qquad\qedsymbol}}

\opn\dis{dis}
\def\pnt{{\raise0.5mm\hbox{\large\bf.}}}

\opn\Lex{Lex}

\begin{document}

 \title{Discrete polymatroids  satisfying a stronger symmetric exchange property}

 \author {Dancheng Lu}

 \begin{abstract}
In this paper we introduce discrete polymatroids satisfying the one-sided strong exchange property and show that they are  sortable (as a consequence their base rings are Koszul) and that they satisfy White's conjecture. Since any pruned lattice path polymatroid satisfies the one-sided strong exchange property, this result provides an alternative proof for  one of the main theorems of J. Schweig in \cite{Sc}, where it is shown that  every pruned lattice path polymatroid  satisfies White's conjecture. In addition  we characterize a class of such polymatroids whose base rings are Gorenstein. Finally   for two classes of pruned  lattice path polymatroidal ideals $I$ and their powers we  determine their depth and their associated prime ideals, and furthermore determine the least power $k$ for which $\depth S/I^k$ and $\Ass(S/I^k)$ stabilize.  It turns out that  $\depth S/I^k$ stabilizes precisely when $\Ass(S/I^k)$ stabilizes  in both cases.
 \end{abstract}

\address{Dancheng Lu, Department of Mathematics, Soochow University, P.R.China} \email{ludancheng@suda.edu.cn}
\thanks{The paper was written while the  author was visiting the Department of Mathematics of
University Duisburg-Essen. He wants to express his thanks for its hospitality.}

 \maketitle

\section*{Introduction}

Throughout this paper, we always denote vectors in boldface such as $\ub,\vb, \ub_i,\vb_i, {\bm \alpha}$ and etc. If $\ub$ is a vector in $\ZZ^n$, we  use either $u_i$ or $\ub(i)$ to denote its $i$th entry and use $\ub(A)$ to denote the number $\sum_{i\in A}u_i$ for a subset $A\subseteq \{1,\ldots,n\}.$  Let $k,\ell$  be integers with $k\leq \ell$. Then $[k,\ell]$ denotes the interval $\{k,k+1,\ldots, \ell\}$ and $[1,k]$ is usually denoted by $[k]$ for short. Also we denote by ${\bm \epsilon}_1,\ldots, {\bm \epsilon}_n$  the canonical basis of $\ZZ^n$ and by $\ZZ_+$ the set of non-negative integers. The set $\ZZ_+^n$ has a partial ordering $\leq$ defined by:
\begin{center}$\ub\leq \vb$ $\Longleftrightarrow$ $u_i\leq v_i$ for each $i=1,\ldots,n.$
\end{center}

Unless otherwise stated, $S$ always stands for the polynomial ring $K[x_1,\ldots,x_n]$ over a field $K$. For a subset $A\subseteq [n]$, $P_A$  denotes the monomial prime ideal $(x_i\:\; i\in A)$ of $S$.

\vspace{2mm}
For the basic knowledge of matroids we refer to \cite{O}. In \cite{HH},  discrete polymatroids are introduced, which generalize matroids in the way that
monomial  ideals generalize squarefree monomial ideals.

\medskip
A {\it discrete polymatroid} on the ground set $[n]$ is a nonempty finite set $\Pb\subseteq \ZZ_+^n$ satisfying
\begin{enumerate}
\item[(D1)] if $\mathbf{u}=(u_1,\ldots,u_n)\in \Pb$ and $\mathbf{v}=(v_1,\ldots,v_n)\in \ZZ_+^n$ with $\mathbf{v}\leq \mathbf{u}$, then $\mathbf{v}\in \Pb$;

\item[(D2)] if $\mathbf{u}=(u_1,\ldots, u_n)\in \Pb$ and $\mathbf{v}=(v_1,\ldots,v_n)\in \Pb$ with $|\mathbf{u}|<|\mathbf{v}|$, then there is $i\in [n]$ with $u_i<v_i$ such that $\mathbf{u}+{\bm \epsilon}_i\in \Pb$. Here $|\mathbf{u}|:=\ub([n])$.

\end{enumerate}

\medskip
A {\it base} of a discrete polymatroid $\Pb$ is a vector $\mathbf{u}$ of $\Pb$ such that $\mathbf{u}<\mathbf{v}$ for no $\mathbf{v}\in \Pb$. Every base of $\Pb$ has the same modulus $\mathrm{rank}(\Pb)$, the {\em rank} of $\Pb$.  Let $B(\Pb)$ or simply $B$ denote the set of bases of $\Pb$.  Every discrete polymatroid satisfies the following symmetric exchange property: if $\mathbf{u}$ and $\mathbf{v}$ are vectors of $B$, then for  any $i\in [n]$ with $u_i<v_i$, there is $j\in [n]$ with $u_j>v_j$ such that both $\mathbf{u}+{\bm \epsilon}_i-{\bm \epsilon}_j$ and $\mathbf{v}-{\bm \epsilon}_i+{\bm \epsilon}_j$ belong to $B$.  Conversely if $B$ is a set of vectors in $\ZZ_+^n$ of the same modulus satisfying the symmetric exchange property, then $\Pb=\{\ub\in \ZZ_+^n\:\; \ub\leq \vb \mbox{\ for some\ } \vb\in B\}$ is a discrete polymatroid with $B$ as its set of bases.

\medskip

We are interested in two algebraic structures associated with a discrete polymatroid $\Pb$: its {\it  base ring}   and its {\it polymatroidal ideal}. Let $K$ be a field. The base ring $K[B(\Pb)]$ (or simply $K[B]$) of $\Pb$ is defined  to be the subring of $K[t_1,\ldots,t_n]$ generated by monomials $\mathbf{t}^\mathbf{u}=t_1^{u_1}\ldots t_n^{u_n}$ with $\mathbf{u}\in B$. Meanwhile the polymatroidal ideal of $\Pb$ is defined to be the  monomial ideal in $S=K[x_1,\ldots,x_n]$ generated by $\mathbf{x}^\mathbf{u}=x_1^{u_1}\ldots x_n^{u_n}$ with $\mathbf{u}\in B$.
\medskip

 Let $T$ be the polynomial ring $K[x_\mathbf{u}:\mathbf{u}\in B]$ and let $I_B$ be the kernel of the $K$-algebra homomorphism $\phi:T\rightarrow K[B]$ with $\phi(x_\mathbf{u})=\mathbf{t}^\mathbf{u}$ for any $\ub\in B$. There are some obvious generators in $I_B$. Indeed,  let $\mathbf{u},\mathbf{v}\in B$ with $u_i>v_i$. Then there exists $j$ such that  $u_j<v_j$  and such that $\mathbf{u}-{\bm \epsilon}_i+{\bm \epsilon}_j$ and $\mathbf{v}+{\bm \epsilon}_i-{\bm \epsilon}_j$ belong to $B$.  We see  that $x_\mathbf{u}x_\mathbf{v}-x_{\mathbf{u}-{\bm \epsilon}_i+{\bm \epsilon}_j}x_{\mathbf{v}+{\bm \epsilon}_i-{\bm \epsilon}_j}\in I_B$. Such relations are  called {\it symmetric  exchange relations.}  White \cite{W} conjectured that for a matroid  the symmetric exchange relations generate $I_B$. In \cite{HH}, Herzog and Hibi predicted  that
this also holds for discrete polymatroids.

\begin{Conjecture*}[White, Herzog-Hibi]
Let $\Pb$ be a discrete polymatroid on the ground set $[n]$ with $B$ as its set of bases. Then $I_B$ is generated by symmetric  exchange relations.
\end{Conjecture*}

We will refer this conjecture as White's conjecture hereafter.  In \cite{HH} it is shown  that if White's conjecture holds for all matroids, then it holds for all discrete polymatroids as well, and that any discrete  polymatroid satisfying the strong symmetric exchange property satisfies  White's conjecture. Recall that a discrete polymatroid $\Pb$ is said to have  the {\it strong symmetric exchange property} if for any bases $\mathbf{u}$ and $\mathbf{v}$ of $\Pb$ with $u_i<v_i$ and $u_j>v_j$,
both $\mathbf{u}+{\bm \epsilon}_i-{\bm \epsilon}_j$ and $\mathbf{v}-{\bm \epsilon}_i+{\bm \epsilon}_j$ are bases of $\Pb$.

\medskip
In \cite{Sc} J. Schweig introduce {\em pruned lattice path polymatroids} and prove that they satisfy White's conjecture.  He actually prove that the symmetric exchange relations form a Gr\"obner basis of $I_B$ for  such discrete polymatroids.

\medskip
In Section~1 we introduce discrete polymatroids satisfying the one-sided strong symmetric exchange property (see Definition~\ref{Def1}) and show that they are  sortable and that they satisfy White's conjecture. It is known by \cite[Lemma 5.2]{HH}  that the sorting relations form a Gr\"obner basis of the defining ideal $I_B$ in this case. As a consequence,  the base ring of  a discrete polymatroid satisfying the one-sided strong symmetric exchange property is Koszul.

\medskip
A pruned lattice path polymatroid may be described  in terms of inequalities of the  set of its bases as follows:

\begin{Definition*}
\label{defpruned}
\em Let $n,d$ be positive integers, and given vectors  $\mathbf{a}, \mathbf{b}, {\bm \alpha},{\bm \beta}$ in $\ZZ_+^n$ such that $\ab\leq \bb$, ${\bm \alpha}\leq {\bm \beta}$ and $\alpha_1\leq \cdots\leq \alpha_n=d$,
 $\beta_1\leq \cdots\leq \beta_n=d$.  A discrete polymatroid $\Pb$ on the ground set $[n]$ is called a {\em pruned lattice path polymatroid} or simply a {\em PLP-polymatroid} (of type $(\mathbf{a},\mathbf{b}|{\bm \alpha},{\bm \beta})$), if the set $B$ of its bases consists of vectors $\ub\in \ZZ_{+}^n$ such that
 \[
a_i \leq u_i\leq b_i  \quad \text{for $i=1,\ldots,n$},
\]
and
\[
  \alpha_{i} \leq u_1+u_{2}+\cdots +u_i\leq \beta_{i} \quad \text{for $i=1,\ldots,n$}.
\]
\end{Definition*}

If $a_i=0$ and $b_i=d$ for $i=1,\ldots, n$, then the first $n$ inequalities can be dropped in this definition. In this case, $\Pb$ is  nothing but the {\em lattice path polymatroid} ({\em LP-polymatroid} for short) discussed in \cite{Sc} and \cite{Sc1}.

\vspace{2mm}
   Taking advantage of this definition  we   prove in Section~2 that a PLP-polymatroid satisfies the two-sided strong symmetric exchange property and thus satisfies White's Conjecture.  We find an example of discrete polymatroid satisfying the two-sided strong symmetric exchange property which is not a PLP-polymatroid. However this example is isomorphic to a PLP-polymatroid. On the other hand, we show for some special discrete polymatroids  that  the one-sided strong  exchange property implies the property of being a PLP-polymatroid. The precise relationship between  polymatroids satisfying the two-sided strong symmetric exchange property and  pruned lattice path polymatroid remains to be revealed.

   \vspace{2mm}
  We have known that the base ring $K[B(\MP)]$ of  every discrete polymatroid $\MP$ is always normal, see e.g. \cite[Theorem 12.5.1]{HHBook} and thus Cohen-Macaulay. It is then natural to ask when those rings are Gorenstein. However it seems quite difficult to  obtain a perfect answer to this problem. In \cite{DH}, the Gorenstein algebra of Veronese  type was classified. Note that the algebra of Veronese type is the base ring of a  PLP-polymatroid of type $(\mathbf{0},\mathbf{b}|\mathbf{0}, (d,\ldots,d))$ for some $\bb\in \ZZ_+^n$ and $0<d\in \ZZ_+$.   In \cite{HH},  generic discrete polymatroids were  introduced and all such discrete polymatroids whose base rings are Gorenstein  were characterized.  In Section~3, we will give a  characterization of  a special  class of PLP-polymatroids  which have  Gorenstein base rings.

\medskip
From Section~4 on, we will turn to investigate the algebraic  properties of  polymatroidal ideals for some certain  PLP-polymatroids.

\medskip
In Section 4, we deduce  a formula to compute the depth for a PLP-polymatroidal ideal, see (\ref{nu}). This formula plays a key role in the last two sections.
 As an immediate application, we determine in Proposition~\ref{lattice} the associated prime ideals of a LP-polymatroidal ideal. This result  will be repeatedly  used in what follows.

\medskip

In the remaining two sections we consider special classes of PLP-polymatroidal ideals, where the questions concerning depth and associated prime ideals of powers of ideals have  complete answers.

\medskip
The ideals considered in Section~5 are called left PLP-polymatroidal ideals. We say that a PLP-polymatroidal ideal of type $(\ab,\bb|{\bm \alpha},{\bm \beta})$ is  a {\em left PLP-polymatroidal ideal}, if  there exists $k\in [n-1]$ such that $a_i=0$, $b_i\geq d$ for all $k+1 \leq i\leq n$ and $\alpha_i=0$, $\beta_i=\beta_{k+1}$ for $1\leq i\leq k$. We first show in Proposition~\ref{depthleft} that $\depth S/I=|\{k+1\leq i\leq n-1\:\;\alpha_i=\beta_i\}|$, and  thus the depth function $\depth S/I^k$ stabilizes from the very beginning.

\medskip
The main tool to determine  the associated prime ideals of our ideals  is {\em monomial localization}. By using this technique it suffices  to characterize  when a suitable  localization of the ideal has  depth zero. This can be checked with formula~(\ref{nu}). Given a subset $A$ of $[1,n]$.  It turns out, see Corollary~\ref{possible}, that $P_A\in \mathrm{Ass}(S/I)$ only if  $A$ is  an interval contained in $[k+2,n]$ or $A=B\cup [k+1,n]$ for some subset $B$ of $[k]$ with $k$ as in the definition of left PLP-polymatroidal ideals. In Theorem~\ref{howtowrite} the precise set of associated prime ideals of $S/I$ is described. As a consequence we obtain in Corollary~\ref{astableft} that  all powers of a left PLP-polymatroidal ideal have the same set of associated prime ideals.

\medskip

In Section~6, we consider {\em right PLP-polymatroidal ideals.} A PLP-polymatroidal ideal of type $(\ab,\bb|{\bm \alpha},{\bm \beta})$ is called a right PLP-polymatroidal ideal, if  there exists $k\in [n-1]$ such that $a_i=0$ and  $b_i\geq d$ for all $1\leq i\leq k$, and $\alpha_i=\alpha_k$ and  $\beta_i=d$ for all $k+1\leq i\leq n-1$. To determine the depth and the associated prime ideals for this class of ideals is  more complicated than  in the case of left one.  For example, if $I$ is a right PLP-polymatroidal ideal  then $\Ass(S/I^k)$ is not stable from the very beginning. The precise power when this happens is given in Theorem~\ref{astabright}. In Proposition~\ref{assright2} the set $\Ass^\infty(S/I)$, which describes the set of associated prime ideals of all large powers of $I$,  is determined. A formula  for the depth of $I$  is given  in Theorem~\ref{depth2} and the powers for which the depth stabilizes is given in Corollary~\ref{dstabright}.

\medskip As observed in \cite{HV},  every polymatroidal ideal is of strong intersection type, that is, it is the intersection of some powers of  its associated  prime ideals. Thus  we can use the results obtained in Section 5 and Section 6 to provide  an  irredundant primary decomposition for any  left or right PLP-polymatroidal ideal.

\medskip
Since the left and right pruned  lattice path polymatroidal ideals have shown such a different algebraic behaviour it is not expected that there is a nice and uniform description of the depth and the set of associated prime ideals for arbitrary PLP-polymatroidal ideals.

\section{The one-sided strong symmetric exchange property and the conjecture of White}

In this section  we introduce the concept of the one-sided strong symmetric exchange property for discrete polymatroids and show that such discrete  polymatroids are sortable and  satisfy White's conjecture.

\begin{Definition}
\label{Def1}
{\em Let $\Pb$ be a discrete polymatroid on the ground set $[n]$ with $B$ as its set of bases. Then we say that $B$  (or $\Pb$) satisfies the {\em left-sided strong symmetric exchange property}, if for any pair $\ub,\vb\in B$ such that  $\ub(i)>\vb(i)$ and  $\ub(1)+\cdots+\ub(i-1)<\vb(1)+\cdots+\vb(i-1)$, there exists  $j\leq i-1$  such that $\ub(j)<\vb(j)$ and  both $\ub-{\bm \epsilon}_i+{\bm \epsilon}_j$ and $\vb+{\bm \epsilon}_i-{\bm \epsilon}_j$ belong to  $B$. }
\end{Definition}

Similarly, the right-sided strong symmetric exchange property is defined. If $B$ or $\Pb$ satisfies both the  left-sided and right-sided strong symmetric exchange property,  we say that  $B$ or $\Pb$ satisfies the two-sided strong symmetric exchange property. Recall that  a discrete polymatroid $\Pb$ is said to have  the {\it strong symmetric exchange property} if for any bases $\mathbf{u}$ and $\mathbf{v}$ of $\Pb$ with $u_i<v_i$ and $u_j>v_j$,
both $\mathbf{u}+{\bm \epsilon}_i-{\bm \epsilon}_j$ and $\mathbf{v}-{\bm \epsilon}_i+{\bm \epsilon}_j$ are bases of $\Pb$. Hence  the two-sided strong symmetric exchange property does not imply the strong symmetric exchange property
\medskip

We also say that a monomial ideal $I$ satisfies the  left-, right- or two-sided strong symmetric exchange property, if it is the polymatroidal ideal of a discrete polymatroid which has this property.

\medskip
In the following theorem we will show if  $B$ satisfies one-sided (left-sided or right-sided) strong symmetric exchange property, then $B$ is sortable and $I_B$ is generated by symmetric exchange relations.

\medskip
For the proof of this result we need some preparations. First we recall the notion of sortability, which was introduced by
Sturmfels \cite{St}.

\vspace{2mm}

 Let $\ub,\vb\in B$, and write $\tb^{\ub}\tb^{\vb}=t_{i_1}t_{i_1}\ldots t_{i_{2d}}$ with $i_1\leq i_2\leq \ldots \leq i_{2d}$. Here $d$ is the rank of $\Pb$.  Then we set $\tb^{\ub'}=t_{i_1}t_{i_3}\ldots t_{i_{(2d-1)}}$ and $\tb^{\vb'}=t_{i_2}t_{i_4}\ldots t_{i_{(2d)}}$. This defines a map:
 $$\mathrm{sort}: B\times B\rightarrow M_d\times M_d, \qquad (\ub,\vb)\rightarrow (\ub',\vb'),$$ where   $M_d$ denotes all vectors $\ub$ in $\ZZ_+^n$ with $|\ub|=d$.

 \vspace{2mm}

 The map ``sort" is called a {\em sorting operator}. A pair $(\ub,\vb)$ is called {\em sorted} if $\mathrm{sort}(\ub,\vb)=(\ub,\vb)$ and  $B$ is called {\em sortable} if for all pair $(\ub,\vb)\in B\times B$, one has $\mathrm{sort}(\ub,\vb)\in B\times B$.   We see that if $(\ub,\vb)$ is sorted, then $\ub-\vb$  is a vector with entries $\pm1$ and $0$.

 \vspace{2mm}

 Let $\ub,\vb$ be elements  in $B$ with $|\ub(i)-\vb(i)|\leq 1$ for $i=1,\ldots,n$. As in  \cite[Page 253 ]{HH} one  associates $(\ub,\vb)$ with a sequence  $s(\ub,\vb)$ of signs $+$ and $-$ only depending  on $\ub-\vb$: reading entries of  $\ub-\vb$ from the left to right we put the sign $+$ or the sign $-$ if we reach the entry $+1$ or entry $-1$.  For example, if $\ub-\vb=(0,0,-1,0,1,1,-1,0)$, then $s(\ub, \vb)= -,+,+,-$. Note that the sequence $s(\ub,\vb)$ always contains  as many  $+$ as $-$ signs, since $\sum_{i=1}^n\ub(i)= \sum_{i=1}^n\vb(i)$ and  $|\ub(i)-\vb(i)|\leq 1$ for $i=1,\ldots,n$. By \cite[Lemma 5.1]{HH}, a pair $(\ub,\vb)$ is sorted if and only if  $s(\mathbf{u},\mathbf{v})$ is a sequence of alternating signs: either  $+,-,+,-,\cdots$ or $-,+,-,+,\cdots$.

\vspace{1mm}

We observe the following  fact: if $\ub(i)-\vb(i)=1$ and $\ub(j)-\vb(j)=-1$, then
$s(\ub-{\bm \epsilon}_i+{\bm \epsilon}_j,\vb+{\bm \epsilon}_i-{\bm \epsilon}_j)$ is obtained from $s(\ub,\vb)$ by exchanging corresponding signs. For instance for $\ub$ and $\vb$ with  $\ub-\vb$  as before, we get $s(\ub+{\bm \epsilon}_3-{\bm \epsilon}_5, \vb-{\bm \epsilon}_3+{\bm \epsilon}_5)=+,-,+-$ which is   obtained from $s(\ub,\vb)$ by exchanging the first and second signs in the sequence.

\vspace{1mm}
Finally, if $\ub,\vb\in B$ and $\mathrm{sort}(\ub,\vb)=(\ub',\vb')$, then $x_{\ub}x_{\vb}-x_{\ub'}x_{\vb'}$ is called a {\em sorting relation}. Note that if $(\ub',\vb')\notin B\times B$, then the sorting relation $x_{\ub}x_{\vb}-x_{\ub'}x_{\vb'}$ does not belong to $I_B$. In view of  \cite[Lemma 5.2]{HH}, we see that if $B$ is sortable then $I_B$ has a Gr\"obner base consisting of the sorting relations: $x_{\ub}x_{\vb}-x_{\ub'}x_{\vb'}$, where $(\ub,\vb)$ ranges over $B\times B$.

\begin{Theorem} Assume that $B$ satisfies the one-sided (left-sided or right-sided) strong symmetric exchange property. Then:

 $\mathrm{(a)}$ $B$ is sortable and sorting relations form a Gr\"obner base of $I_B$;

$ \mathrm{(b)}$  $I_B$ is generated by symmetric exchange relations.
\end{Theorem}

  \begin{proof} (a)  Without loss of generality we assume that $B$ satisfies the right-sided  strong symmetric exchange property. We denote by $E_B$ the ideal (contained in $I_B$) which is generated by the symmetric exchange relations.

\vspace{2mm}
  Let $(\mathbf{u},\mathbf{v})\in B\times B$. By \cite[Lemma 5.4]{HH} there exists  $(\mathbf{u}_1,\mathbf{v}_1)\in B\times B$ such that $x_\mathbf{u}x_\mathbf{v}-x_{\mathbf{u}_1}x_{\mathbf{v}_1}\in E_B$ and $|\mathbf{u}_1(i)-\mathbf{v}_1(i)|\leq 1$ for  $i=1,\ldots,n$.  We claim that there exists  $(\mathbf{u}_2,\mathbf{v}_2)\in B\times B$ such that $x_{\mathbf{u}_1}x_{\mathbf{v}_1}-x_{\mathbf{u}_2}x_{\mathbf{v}_2}\in E_B$ and $(\mathbf{u}_2,\mathbf{v}_2)$ is sorted.

\vspace{1mm}
The sign at position $i$ of the sequence $s(\mathbf{u}_1,\mathbf{v}_1)$ will be denoted by $s_i(\mathbf{u_1},\mathbf{v_1})$. Without loss of generality we may suppose that  $s_1(\mathbf{u}_1,\mathbf{v}_1)=+$. Let $i\geq 2$ be the smallest integer with the property that $s_{i-1}(\mathbf{u}_1,\mathbf{v}_1)=s_i(\mathbf{u}_1,\mathbf{v}_1)$. For convenience, we denote by $c(\mathbf{u}_1,\mathbf{v}_1)$ the number $i$   and set $c(\mathbf{u}_1,\mathbf{v}_1)=\infty$ if no such $i$ exists. Then,  $(\ub_1,\vb_1)$ is sorted if and only if $c(\mathbf{u}_1,\mathbf{v}_1)=\infty$, and  $c(\mathbf{u}_1,\mathbf{v}_1)<n$ if $c(\mathbf{u}_1,\mathbf{v}_1)<\infty$.
\vspace{2mm}

We consider the following cases.
\vspace{2mm}

 If $c(\mathbf{u}_1,\mathbf{v}_1)=\infty$, then $(\ub_1,\vb_1)$ is sorted and there is nothing to prove.

\vspace{1mm}
 Assume that $c(\mathbf{u}_1,\mathbf{v}_1)<\infty$.  Then we let $i=c(\mathbf{u}_1,\mathbf{v}_1)$. Suppose first that $s_i(\mathbf{u}_1,\mathbf{v}_1)=+$. By definition,  $s_i(\mathbf{u}_1,\mathbf{v}_1)$ corresponds to the sign  of  the  $i_1$-th entry of $\mathbf{u}_1-\mathbf{v}_1$ for some $i_1\geq i$. Note that our assumptions imply that  $\mathbf{u}_1(i_1)-\mathbf{v}_1(i_1)=1$ and  $\mathbf{u}_1(i_1+1)+\cdots+\mathbf{u}_1(n)<\mathbf{v}_1(i_1+1)+\cdots+\mathbf{v}_1(n)$. Therefore, since  $B$ satisfies the right-sided strong symmetric exchange property,   there exist $j_1>i_1$ such that $\mathbf{u}_1(j_1)-\mathbf{v}_1(j_1)=-1$ and that both $\mathbf{u}_1-{\bm \epsilon}_{i_1}+{\bm \epsilon}_{j_1}$ and $\mathbf{v}_1+{\bm \epsilon}_{i_1}-{\bm \epsilon}_{j_1}$ belong to  $B$. Since $s(\mathbf{u}_1-{\bm \epsilon}_{i_1}+{\bm \epsilon}_{j_1},\mathbf{v}_1+{\bm \epsilon}_{i_1}-{\bm \epsilon}_{j_1})$ is obtained from $s(\mathbf{u}_1,\mathbf{v}_1)$ by exchanging $s_i(\mathbf{u}_1,\mathbf{v}_1)$ and  $s_j(\mathbf{u}_1,\mathbf{v}_1)$ for suitable $j>i$, it follows that  $c(\mathbf{u}_1-{\bm \epsilon}_{i_1}+{\bm \epsilon}_{j_1},\mathbf{v}_1+{\bm \epsilon}_{i_1}-{\bm \epsilon}_{j_1})\geq i+1$.

\vspace{1mm}
  Suppose next that $s_i(\mathbf{u}_1-\mathbf{v}_1)=-$. Similarly as the case above, let $i_1$ be the entry of $\mathbf{u}_1-\mathbf{v}_1$ corresponding to $s_i(\mathbf{u}_1,\mathbf{v}_1)$. Then, since  $\mathbf{u}_1(i_1)-\mathbf{v}_1(i_1)=-1$ and  $\mathbf{u}_1(i_1+1)+\cdots+\mathbf{u}_1(n)>\mathbf{v}_1(i_1+1)+\cdots+\mathbf{v}_1(n)$, the right-sided symmetric exchange property implies that there exists  $j_1>i_1$ such that $\mathbf{u}_1(j_1)-\mathbf{v}_1(j_1)=1$ and that $\mathbf{u}_1+{\bm \epsilon}_{i_1}-{\bm \epsilon}_{j_1}$ and $\mathbf{v}_1-{\bm \epsilon}_{i_1}+{\bm \epsilon}_{j_1}$ belong to  $B$. Again, we have $c(\mathbf{u}_1+{\bm \epsilon}_{i_1}-{\bm \epsilon}_{j_1}, \mathbf{v}_1-{\bm \epsilon}_{i_1}+{\bm \epsilon}_{j_1})\geq i+1$.

\vspace{2mm}
Thus in both cases we obtain a pair $(\mathbf{u}',\mathbf{v}')\in B\times B$ such that $x_{\mathbf{u}_1}x_{\mathbf{v}_1}-x_{\mathbf{u}'}x_{\mathbf{v}'}\in E_B$ and $c(\mathbf{u}',\mathbf{v}')>c(\mathbf{u}_1,\mathbf{v}_1)$. Hence the claim follows by induction on $n-c(\ub_1,\vb_1)$.

\vspace{1mm}
Now let $(\mathbf{u}_2,\mathbf{v}_2)$ be as in the claim. Since $x_{\mathbf{u}_1}x_{\mathbf{v}_1}-x_{\mathbf{u}_2}x_{\mathbf{v}_2}\in E_B$, we have $\mathbf{u}_1+\mathbf{v}_1=\mathbf{u}_2+\mathbf{v}_2$, and so $\mathrm{sort}(\mathbf{u}_1,\mathbf{v}_1)=\mathrm{sort}(\mathbf{u}_2,\mathbf{v}_2)$. But $(\mathbf{u}_2,\mathbf{v}_2)$ is sorted, hence $\mathrm{sort}(\mathbf{u},\mathbf{v})=(\mathbf{u}_2,\mathbf{v}_2)\in B\times B$.  This implies $B$ is sortable and so $I_B$ has a Gr\"obner base consisting of sorting relations by \cite[Lemma 5.2]{HH}.
\vspace{2mm}

(b)  From the proofs of (a), we see that for any $(\mathbf{u},\mathbf{v})\in B\times B$, we have $x_\mathbf{u}x_\mathbf{v}-x_{\mathbf{u}'}x_{\mathbf{v}'}\in E_B$, where $(\mathbf{u}',\mathbf{v}')=\mathrm{sort}(\mathbf{u},\mathbf{v})$.  This  implies that all sorting relations belong to  $E_B$. Hence  $I_B=E_B$, as required.
\end{proof}

\section{Pruned Lattice path Polymatroids}

In this section we will give an alternative definition of a pruned lattice path polymatroid and show that  this class of discrete polymatroids satisfies the two-sided strong symmetric exchange property.
 In addition, we will present some  basic properties of this class of discrete polymatroids.

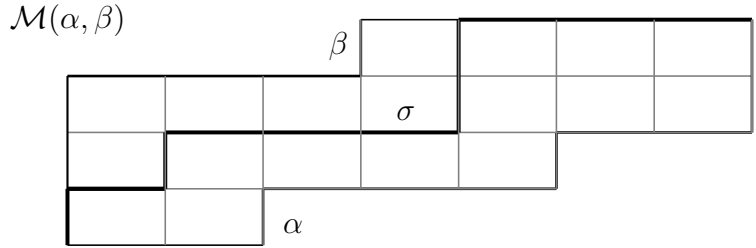
\begin{figure}[ht!]
\begin{center}
\psset{unit=1cm}
\begin{pspicture}(2.75,0.0)(8,5)

\psline[linewidth=0.6pt,linecolor=black](3.9, 5.25)(9.1,5.25)

\psline[linewidth=0.6pt,linecolor=gray](0, 4.50)(9.1,4.50)

\psline[linewidth=1pt,linecolor=black](0, 4.50)(3.9,4.50)

\psline[linewidth=1pt,linecolor=black](3.9,4.50)(3.9,5.25)

\psline[linewidth=1pt,linecolor=black](2.6, 2.25)(2.6,3.00)

\psline[linewidth=1pt,linecolor=black](2.6,3.00)(6.5,3.00)
\psline[linewidth=1pt,linecolor=black](6.5,3.00)(6.5,3.75)

\psline[linewidth=1pt,linecolor=black](6.5,3.75)(9.1, 3.75)
\psline[linewidth=1pt,linecolor=black](9.1, 3.75)(9.1,5.25)

\psline[linewidth=0.6pt,linecolor=gray](0, 3.75)(9.1,3.75)

\psline[linewidth=0.6pt,linecolor=gray](0, 3.00)(6.5,3.00)

\psline[linewidth=1.5 pt,linecolor=black](0, 2.25)(0,3.00)

\psline[linewidth=1.5pt,linecolor=black](0, 3.00)(1.3,3.00)
\psline[linewidth=1.5pt,linecolor=black](0, 3.00)(1.3,3.00)
\psline[linewidth=1.5pt,linecolor=black](1.3,3.00)(1.3,3.75)

\psline[linewidth=1.5pt,linecolor=black](1.3,3.75)(5.2,3.75)

\psline[linewidth=1.5pt,linecolor=black](5.2,3.75)(5.2, 5.25)

\psline[linewidth=1.5pt,linecolor=black](5.2, 5.25)(9.1, 5.25)

\psline[linewidth=0.6pt,linecolor=black](0, 2.25)(2.6,2.25)

\psline[linewidth=0.8pt,linecolor=black](0, 2.25)(0, 4.50)

\psline[linewidth=0.6pt,linecolor=gray](1.3, 2.25)(1.3, 4.50)

\psline[linewidth=0.6pt,linecolor=gray](2.6, 2.25)(2.6, 4.50)

\psline[linewidth=0.6pt,linecolor=gray](3.9, 3.00)(3.9, 5.25)

\psline[linewidth=0.6pt,linecolor=gray](5.2, 3.00)(5.2, 5.25)

\psline[linewidth=0.6pt,linecolor=gray](6.5, 3.00)(6.5, 5.25)
\psline[linewidth=0.6pt,linecolor=gray](7.8, 3.75)(7.8, 5.25)
\psline[linewidth=0.8pt,linecolor=gray](9.1, 3.75)(9.1, 5.25)

\rput(5,1){ $N(\alpha)=

\{3,7,10,11\}$ and $N(\beta)=\{1,2,3,7\}$ }

\rput(5,1.6){$(n,r)=(8,5)$}

\rput(5,0.4){$E_{\sigma}=(2,3,3,3,5,5,5)$ and $m(\sigma)=x_1x_2x_5^2$}

\rput(3.6, 4.9){$\beta$}
\rput(0, 5.25){$\mathcal{M}(\alpha,\beta)$}

\rput(4.5,4.0){$\sigma$}

\rput(3,2.5){$\alpha$}
\end{pspicture}
\end{center}
\caption{ \ }

\label{Fig1}
\end{figure}

First of all, we recall from \cite{Sc} some definitions and some facts about lattice path polymatroids. Fix two integers $n, r\geq 1$. A {\it  lattice path} is a sequence of unit-length steps in the plane, each either due
north or east, beginning at the point $(1,1)$ and ending at the point $(n, r)$. Note that a lattice path in the original definition of \cite{Sc} begins at the original point $(0,0)$. We make such an adaption because we want to  consider monomials in the polynomial ring $K[x_1,\ldots,x_n]$ (not $K[x_0,\ldots,x_n]$). In our setting, it needs $n+r-2$ steps from the beginning to the end.  Fix a lattice path $\sigma$. Define the subset $N(\sigma) \subseteq [n+r-2]$  by the following rule:
\begin{center}$i\in N(\sigma)\quad \Longleftrightarrow$\quad the $i$th step of $\sigma$ is north.

 \end{center}

 Similarly, $E(\sigma)$ is the subset of $[n+r-2]$ defined by $i\in E(\sigma)$ $\Leftrightarrow$ the $i$th step of $\sigma$ is east. Hence $$E(\sigma)=[n+r-2]\setminus N(\sigma).$$  Also  the vectors $N_{\sigma}$ and $E_{\sigma}$ are defined  as follows: the $i$th entry of $N_{\sigma}$ is  the vertical coordinate of the $i$th north step of $\sigma$, and   Similarly, the $i$th entry of $E_{\sigma}$ is the horizontal coordinate  of the $i$th east step of $\sigma$. Thus,  if $E(\sigma)=\{a_1,\ldots,a_{n-1}\}$ with $a_1<a_2<\cdots<a_{n-1}$, then $$E_{\sigma}=(a_1,a_2-1, \ldots,a_{n-1}-(n-2)).$$ Similarly, if $N(\sigma)=\{b_1,\ldots,b_{r-1}\}$ with $b_1<b_2<\cdots<b_{r-1}$, then $N_{\sigma}=(b_1,b_2-1,\ldots,b_{r-1}-(r-2))$. Let $m(\sigma)$ be the monomial $x_1^{u_1}\cdots x_n^{u_n}$, where $u_i$ is the degree of $x_i$ is the number of north steps along the vertical line $x=i$. Hence,  if $m(\sigma)=x_1^{u_1}\cdots x_n^{u_n}$ and $E_{\sigma}=(a_1,\ldots,a_n)$, then
 \begin{center} $u_1+\cdots+u_i=a_{i}-1$ for $i=1,\ldots,n-1$ and $u_1+\cdots+u_n=r-1$.
 \end{center}

Let $\sigma$ and $\tau$ be lattice paths. We say that $\sigma$ is above $\tau$ if $E_{\sigma}\geq E_{\tau}$, that is, the $i$th entry of $E_{\sigma}$ is greater than or equal to the  $i$th one of $E_{\tau}$ for $1\leq i\leq n-1$. In this case, we write $\sigma \succeq \tau$ (or $\tau \preceq \sigma$).
Now fix two lattice paths $\alpha$ and $\beta$ with $\alpha\preceq \beta$ and set  $$\mathcal{M}(\alpha,\beta):=\{m(\sigma)\:\;\alpha\preceq \sigma\preceq \beta)\}.$$ Then $(\mathcal{M}(\alpha,\beta))$, the  ideal of $ K[x_1,\ldots,x_n]$ generated by monomials in $\mathcal{M}(\alpha,\beta)$, is a polymatroidal ideal, since the  exponent set $$\log\mathcal{M}(\alpha,\beta):= \{\ub\:\; \xb^{\ub}\in \mathcal{M}(\alpha,\beta)\}$$  is the set of bases of a discrete polymatroid, which is called a {\em lattice path polymatroid}. Assume that $E_{\alpha}=(\alpha_1,\ldots,\alpha_{n-1})$ and $E_{\beta}=(\beta_1,\ldots,\beta_{n-1})$.  Then, for a vector $\ub\in \ZZ_+^n$, $\xb^{\ub}\in \mathcal{M}(\alpha,\beta)$ if and only if $$\alpha_i-1\leq u_1+\cdots+u_i\leq \beta_i-1 \mbox{\ for \ }i=1,\cdots,n-1 \mbox{\ and\ } u_1+\cdots+u_n=r-1.$$

 Recall that a discrete polymatroid is {\em transversal} if its polymatroidal ideal is the product of some monomial prime ideals. It was observed in \cite{Sc}  that every lattice path polymatroid is transversal.  More exactly, if we denote by  $I$  the polymatroidal ideal of a lattice path polymatroid whose set of bases is $\log \mathcal{M}(\alpha,\beta)$ and assume that $N_{\beta}=(s_1,\ldots,s_d)$ and $N_{\alpha}=(t_1,\ldots,t_d)$, then $I=P_{[s_1,t_1]}\cdots P_{[s_d,t_d]}.$  Its converse statement is also true by observing carefully: namely, if $I=P_{[s_1,t_1]}\cdots P_{[s_d,t_d]}$ with $s_1\leq \ldots \leq s_d$, $t_1\leq \ldots \leq t_d$ and $s_i\leq t_i$ for $i=1,\ldots,d$, then $I$ is the polymatroidal ideal of some lattice path polymatroid. In conclusion we have

 \begin{Lemma}

  \label{trans1} Let $I$ be a monomial ideal generated in degree $d$. Then $I$ is a lattice path  polymatroidal ideal  if and only if $I=P_{[s_1,t_1]}\cdots P_{[s_d,t_d]}$ for some $s_i,t_i$ satisfying $s_1\leq \ldots \leq s_d$, $t_1\leq \ldots \leq t_d$ and $s_i\leq t_i$ for $i=1,\ldots,d$.

 \end{Lemma}

 \begin{Example}{\em Let $\Pb$ be the lattice path polymatroid displayed in Figure~\ref{Fig1}. Then $E_{\alpha}=(1,1,2,2,2,3,3)$ and $E_{\beta}=(4,4,4,5,5,5,5).$  Denote by $\alpha_i$ the $i$-th entry of $E_{\alpha}$, and by $\beta_i$ the  $i$-th entry of $E_{\beta}$ for $i=1,\ldots,7$.   Then $\ub\in \ZZ_+^8$ is a base of $\Pb$ if and only if
 \begin{center}
   $\alpha_i-1\leq u_1+\cdots+u_i\leq \beta_i-1$ for $i=1,\ldots,7$ and $u_1+\cdots+u_8=4$.
\end{center}

Let  $I$ be the polymatroidal ideal of $\Pb$.
Since $N_{\beta}=(1,1,1,4)$ and $N_{\alpha}=(3,6,8,8)$, we have
$$I=P_{[1,3]}P_{[1,6]}P_{[1,8]}P_{[4,8]}.$$}

\end{Example}
Let us recall the concept of a {\em pruned discrete polymatroid} given in \cite{Sc}.
 Suppose that $B$ is the set of bases of a discrete polymatroid $\Pb$ on the ground set $[n]$ and that $\bb=(b_1,\ldots,b_n)$ is a vector of $\ZZ_+^n$. Then
$$B_{\bb}=\{\ub\in B\:\; u_i\leq b_i \mbox{\ for\ } 1\leq i\leq n\}$$
is again the set of bases of a discrete polymatroid, which is denoted by $\Pb_{\bb}$. We call $\Pb_{\bb}$ a {\em pruned  discrete polymatroid} of $\Pb$.

\medskip

 Thus,  a {\em pruned lattice path polymatroid}, namely, a pruned  discrete polymatroid of a lattice path polymatroid, can also be defined as follows.
\begin{Definition}
\label{defpruned}
\em Let $n,d$ be positive integers, and given vectors  $\mathbf{a}, \mathbf{b}, {\bm \alpha},{\bm \beta}$ in $\ZZ_+^n$ such that $\ab\leq \bb$, ${\bm \alpha}\leq {\bm \beta}$ and $\alpha_1\leq \cdots\leq \alpha_n=d$,
 $\beta_1\leq \cdots\leq \beta_n=d$.  A discrete polymatroid $\Pb$ on the ground set $[n]$ is called a {\em pruned path lattice polymatroid} or simply {\it a PLP-polymatroid }  (of type $(\mathbf{a},\mathbf{b}|{\bm \alpha},{\bm \beta})$), if the set $B$ of its bases consists of vectors $u\in \ZZ_{+}^n$ such that \begin{eqnarray}
\label{eq1}
a_i \leq u_i\leq b_i  \quad \text{for $i=1,\ldots,n$},
\end{eqnarray}
and
\begin{eqnarray}
\label{eq2}
  \alpha_{i} \leq u_1+u_{2}+\cdots +u_i\leq \beta_{i} \quad \text{for $i=1,\ldots,n$}.
\end{eqnarray}
\end{Definition}

  We will explain this definition a bit. If  $\alpha_i=0$ and  $\beta_i=d$ for all $i\leq n-1$, then the first $(n-1)$ inequalities in (2) can be dropped and so    $\Pb$ is a discrete polymatroid of Veronese type.
If all $a_i=0$ and all $b_i\geq d$, then all inequalities in (1) can be dropped and $\Pb$ is a lattice path polymatroid; If this is the case, we say that $\Pb$ is a lattice path polymatroid of type $({\bm \alpha},{\bm \beta})$ or a LP-polymatroid of type $({\bm \alpha},{\bm \beta})$ for short.

\medskip
We also say that a monomial ideal $I$ is a  PLP-polymatroidal ideal (of type $(\mathbf{a},\mathbf{b}|{\bm \alpha},{\bm \beta})$) if it is the polymatroidal ideal of a PLP-polymatroid (of type $(\mathbf{a},\mathbf{b}|{\bm \alpha},{\bm \beta})$) and that a monomial ideal $I$ is a  LP-polymatroidal ideal (of type $({\bm \alpha},{\bm \beta})$) if it is the polymatroidal ideal of a LP-polymatroid (of type $({\bm \alpha},{\bm \beta})$).

\medskip
As an example,  consider the  graphic matroid $\MM$  of the graph $G$ as shown in Figure~\ref{Fig2}.
\begin{figure}[hbt]
\begin{center}
\psset{unit=1cm}
\begin{pspicture}(2.75,1.5)(8,5)

\rput(6,2.25){$\bullet$}
\rput(7.5,2.25){$\bullet$}
\rput(6,3.75){$\bullet$}
\rput(7.5,3.75){$\bullet$}

\psline[linewidth=0.6pt,linecolor=black](6,2.25)(7.5,2.25)
\psline[linewidth=0.6pt,linecolor=gray](7.5,2.25)(7.5,3.75)
\psline[linewidth=0.6pt,linecolor=gray](6,2.25)(6,3.75)

\psline[linewidth=0.6pt,linecolor=black](6,3.75)(7.5,3.75)

\rput(4.5,2.25){$\bullet$}
\psline[linewidth=0.6pt,linecolor=gray](4.5,2.25)(4.5,3.75)

\rput(4.5,3.75){$\bullet$}
\psline[linewidth=0.6pt,linecolor=black](4.5,2.25)(6,2.25)
\psline[linewidth=0.6pt,linecolor=black](4.5,3.75)(6,3.75)

\rput(3,2.25){$\bullet$}
\psline[linewidth=0.6pt,linecolor=gray](3,2.25)(3,3.75)

\rput(3,3.75){$\bullet$}
\psline[linewidth=0.6pt,linecolor=black](4.5,2.25)(3,2.25)
\psline[linewidth=0.6pt,linecolor=black](4.5,3.75)(3,3.75)

\rput(2.8,3){$e_1$}
\rput(4.3,3){$e_4$}\rput(6.2,3){$e_7$}\rput(7.8,3){$e_{10}$}

\rput(3.75,3.95){$e_2$}\rput(5.25,3.95){$e_5$}\rput(6.75,3.95){$e_8$}

\rput(3.75,2.05){$e_3$}\rput(5.25,2.05){$e_6$}\rput(6.75,2.05){$e_9$}

\rput(5.25,1.6){$\mathrm{G}$}

\end{pspicture}
\end{center}
\caption{ \ }
\label{Fig2}
\end{figure}

We will show that $\MM$  is a PLP-polymatroid. Let $B$ be the set of bases of $\MM$. Every element in $B$ is identified with a 0-1 vector of dimension 10.
 Since a set of $7$ edges of $G$ forms  a spanning tree of $G$ if and only if it does not contain a cycle, a 0-1 vector  $\ub\in \ZZ_+^{10}$  belongs to $B$ if and only if $\ub$  satisfies the following system of inequalities:
$$\sum_{i=1}^4 u_i\leq 3, \quad  \sum_{i=1}^7 u_i\leq 5,\quad \sum_{i=4}^7u_i\leq 3,\quad
\sum_{i=7}^{10}u_i\leq  3,\quad  \sum_{i=4}^{10}u_i\leq 5,\quad  \sum_{i=1}^{10}u_i=7.$$

 One can check that this set of inequalities  is equivalent to the following set of inequalities
$$ 2\leq \sum_{i=1}^3u_i,\quad \sum_{i=1}^4 u_i\leq 3, \quad 4\leq \sum_{i=1}^6 u_i,\quad
\sum_{i=1}^7u_i\leq 5,\quad \sum_{i=1}^{10}u_i=7.$$  This shows that $\MM$ is indeed a PLP-polymatroid.

\medskip
As another example, consider  the transversal ideal $(x_1,x_2)(x_3,x_4)(x_1,x_3)$. It is a monomial ideal generated by $\xb^{\ub}\in k[x_1,\ldots,x_4]$ with $\ub$ satisfying: $0\leq u_i\leq 2$ for $i=1,3$, $0\leq u_i\leq 1$ for $i=2,4$,   $1\leq u_1+u_2\leq 2$ and $u_1+u_2+u_3+u_4=3$.  Hence it is   a PLP-polymatroidal ideal, but not a LP-polymatroidal ideal.

\begin{Proposition}
\label{pro}Let  $\Pb$ be a PLP-polymatroid with $B$ as its set of bases. Then  $\Pb$  satisfies the  two-sided strong symmetric exchange property.
\end{Proposition}

\begin{proof} We only need to prove that $B$ satisfies the left-sided strong symmetric exchange property,  since the other case is treated similarly. Let $\ub,\vb$ be two vectors of $B$ such that $u_i<v_i$, where $1< i\leq n$ and  $u_1+\cdots+u_{i-1}>v_1+\cdots+v_{i-1}$. Then there is $1\leq k\leq i-1$ such that $u_k>v_k$. Let $j$ be the largest number $k$ with this property. We will  show that both $\ub-{\bm \epsilon}_j+{\bm \epsilon}_i$ and $\vb+{\bm \epsilon}_j-{\bm \epsilon}_i$ belong to $B$.  For this, we write $\ub-{\bm \epsilon}_j+{\bm \epsilon}_i=(t_1,t_2,\ldots,t_n)$, that is, $t_k=u_k$ if $k\notin \{j,i\}$ and $t_j=u_j-1$, $t_i=u_i+1$. Since $\Pb$ is a PLP-polymatroid, we may assume that the elements of $B$ satisfy the  inequalities of Definition~\ref{defpruned}.

It is clear that $a_k \leq t_k\leq b_k$  for all $k=1,\ldots,n$.
If $k\leq j-1$ or $k\geq i$, then $t_1+\cdots+t_k=u_1+\cdots+u_k$, and in particular, $\alpha_k \leq t_1+\cdots+t_k\leq \beta_k$. Fix $k\in [j, i-1]$. Then $t_1+\cdots+t_k=u_1+\cdots+u_k-1\leq \beta_k$. Note that $u_{k+1}\leq v_{k+1},\cdots,u_{i-1}\leq v_{i-1}$ by the choice of $j$, it follows that $u_1+\cdots+u_k>v_1+\cdots+v_k$ and so $t_1+\cdots+t_k=u_1+\cdots+u_k-1\geq v_1+\cdots+v_k\geq \alpha_k$. Hence $\ub-{\bm \epsilon}_j+{\bm \epsilon}_i\in B$. Similarly $\vb+{\bm \epsilon}_j-{\bm \epsilon}_i\in B$, as required.
\end{proof}

 \begin{Example} \label{2.4} \em Let $I=(x_1,x_3)(x_2,x_4)$. Then $I$ is a polymatroidal ideal of a discrete polymatroid satisfying the  two-sided strong symmetric exchange property. If $I$ is  a PLP-polymatroidal ideal, then there exist $a_i,b_i,i=1,\ldots,4$ and $\alpha_2,\beta_2$ such that $I$ is generated by monomials $\xb^{\ub}$  with $\ub$ satisfying   $$a_i\leq u_i\leq b_i,i=1,\ldots,4 ,\qquad \alpha_2\leq u_1+u_2\leq \beta_2, \qquad u_1+\cdots+u_4=2.$$  Note that the inequality $\alpha_3\leq u_1+u_2+u_3\leq \beta_3$ does not appear in the conditions above since it is equivalent to the inequality $d-\beta_3\leq u_4\leq d-\alpha_3$. Since $u_1$ can be 1,  it follows that $b_1\geq 1$. Proceeding in this way, we have $a_i=0,b_i\geq 1$ for $i=1,\ldots,4$ and $\alpha_2=0$, $\beta_2\geq 2$. Therefore we have $x_1x_3\in I$, a contradiction. Hence $I$ is not a PLP-polymatroidal ideal.
 \end{Example}

 However,  $I$  is {\em isomorphic} to the  lattice path polymatroidal ideal $(x_1,x_2)(x_3,x_4)$. Till now, we cannot find a discrete polymatroid satisfying the one-sided  symmetric exchange property which  is not isomorphic to a PLP-polymatroid. In \cite{HHV}, it is proved that a discrete polymatroid satisfies the strong symmetric exchange property if and only if it is isomorphic to a discrete polymatroid of Veronese type. In view of these facts, it may be reasonable for us  to have the following conjecture:

 \begin{Conjecture} \label{C1}
Let $\Pb$ be a discrete polymatroid.  Then the following statements are equivalent:

$\mathrm{(1)}$ $\Pb$ is isomorphic to a PLP-polymatroid;

$\mathrm{(2)}$ $\Pb$ satisfies the two-sided strong symmetric exchange property;

$\mathrm{(3)}$ $\Pb$ satisfies the one-sided (left-sided or right-sided) strong symmetric exchange property.

\end{Conjecture}

We are far from proving this conjecture. However,  we could show that  the conjecture is true indeed in two special cases.
\medskip

        To prove these results and for the later use, we  recall some definitions and facts one can find in  \cite{HH} or \cite{HHBook}. A {\em polymatroid} on the ground set $[n]$ is a convex polytope $\MP$ contained in $\RR_+^n$ such that

  (1) if $\ub\in \MP$ and $\vb$ is a vector $\RR_+^n$ with $\vb\leq \ub$, then $\vb\in \MP$,

  (2) if $\ub,\vb\in \MP$ with $|\vb|>|\ub|$, there exists a vector $\wb\in \MP$ such that $\ub<\wb\leq \ub\vee \vb$. Here $\ub\vee \vb=(\max\{u_1,v_1\},\ldots,\max\{u_n,v_n\})$.

  \vspace{2mm}

   The ground set rank function of a polymatroid $\MP$ is the  function: $\rho:2^{[n]}\to \RR_+$ defined by $$\rho(A)=\max\{\ub(A)\:\; \ub\in \MP\}$$ for all $\emptyset\neq A\subseteq [n]$ together with $\rho(\emptyset)=0$.

\vspace{2mm}
     The function $\rho$ is {\em  nondecreasing}, i.e., $\rho(A)\leq \rho(B)$ if $A\subseteq B\subseteq [n]$, and is {\em  submodular}, i.e.,
     $$\rho(A)+\rho(B)\geq \rho(A\cup B)+\rho(A\cap B)$$ for any $A,B\subseteq [n]$.

\vspace{2mm}
     Conversely, if  given a nondecreasing and submodular function $\rho:2^{[n]}\rightarrow \RR_+$, then the set
     $$\{\ub\in \RR_+^n\:\; \ub(A)\leq \rho(A) \mbox{\quad for all\quad } A\in 2^{[n]}\}$$ is a  polymatroid.

\vspace{2mm}

     An  {\em integral  polymatroid}  is a polymatroid for which every vertex is a lattice point, that is, a vector in $\ZZ_+^n$. A polymatroid is integral if and only if its ground set rank function is integer valued. The relation between an integral polymatroid and a discrete polymatroid are as follows:  If $\MP$ is an integral polymatroid then $\MP\cap \ZZ_+^n$ is a discrete polymatroid;  Conversely if $\Pb$ is a discrete polymatroid on the ground set $[n]$,  then $\MP=\conv(\Pb)$, the convex hull of $\Pb$ in $\RR_+^n$, is an integral polymatroid with $\Pb=\MP\cap \ZZ_+^n$. Thus the ground set rank function of an integral polymatroid $\MP$ is determined by its values on its corresponding discrete polymatroid $\Pb=\MP\cap \ZZ_+^n$, namely,  if $\rho$ is the ground set rank function of $\MP$, then $\rho(A)=\max\{\ub(A)\:\; \ub\in \Pb\},$  see \cite[Theorem 3.4]{HH} and its proof.   Furthermore we have the following representation  of $\Pb$ in this case:
     \begin{eqnarray}
     \label{rho}
     \Pb=\{\ub\in \ZZ_+^n\:\;\ub(A)\leq \rho(A) \mbox{\quad for all\quad} A\in 2^{[n]} \}.
\end{eqnarray}

   \medskip

    Let $\rho$ be the ground set rank function of a discrete polymatroid $\Pb$. A subset $\emptyset \neq A\subseteq [n]$ is called {\em $\rho$-closed} if $\rho(A)<\rho(B)$ for any subset $B\subseteq [n]$ which contains $A$ properly and a subset $\emptyset \neq A\subseteq [n]$ is called {\em $\rho$-separable} if there exist nonempty subsets $A_1,A_2$ with $A_1\cap A_2=\emptyset$ and $A_1\cup A_2=A$  such that $\rho(A)=\rho(A_1)+\rho(A_2)$.

    For $i\in [n]$, we let  $$\mathcal{H}_+^i=\{\ub\in \ZZ^n\:\; \ub(i)\geq 0\}.$$ And for $\emptyset\neq A\subseteq [n]$, let $$\mathcal{H}_A^+=\{\ub\in \ZZ^n\:\; \ub(A)\leq \rho(A)\}.$$
 By  \cite[Proposition 7.2]{HH}, a result taken  from \cite{E}, and by  \cite[Theorem B.1.6]{HHBook}, we see that $\Pb$ has the following irredundant decomposition
 \begin{eqnarray}\label{irr}\Pb=(\cap_{i=1}^n \mathcal{H}_+^{i})\cap (\cap_{A}\mathcal{H}_A^{+}),\end{eqnarray} where $A$ ranges through all $\rho$-inseparable and $\rho$-closed subset of $[n]$. Moreover this is the unique irredundant decomposition of $\Pb$ by  \cite[Theorem B.1.7]{HHBook} in the sense that every irredundant decomposition of $\Pb$ as the intersection of closed half-spaces concides with (\ref{irr}). In what follows, if $A=\{i_1,\ldots,i_s\}$ we denote $\rho(A)$ by $\rho(i_1,\ldots,i_s)$.

\begin{Lemma} \label{cri} Let $\Pb$ be a discrete polymatroid on the ground set $[n]$ with the ground set rank function $\rho$. If every $\rho$-closed and $\rho$-inseparable subset of $[n]$ belongs to $T$, then $\Pb$ is a PLP-polymatroid. Here $$T=\{[i]\:\; i\in [n]\}\cup \{[n]\setminus [i]\:\; i\in [n]\}\cup \{\{i\}\:\; i\in [n]\}\cup \{[n]\setminus \{i\}\:\; i\in[n]\}.$$

\end{Lemma}

\begin{proof} Let $T_1$ be the set of $\rho$-closed and $\rho$-inseparable subset of $[n]$. In view of Equation (\ref{irr}) we have $$\Pb=\{\ub\in \ZZ_+^n\:\; \ub(A)\leq \rho(A) \mbox{\ for any\ } A\in T_1\}.$$
This together with Equation (\ref{rho}) implies $\Pb=\{\ub\in \ZZ_+^n\:\; \ub(A)\leq \rho(A) \mbox{\ for any\ } A\in T\}$ and so $$B(\Pb)=\{\ub\in \ZZ_+^n\:\; \ub(A)\leq \rho(A) \mbox{\ for any\ } A\in T \mbox{\ and } \ub([n])=\rho([n])\}$$

Set $d=\rho([n])$. For $i=1,\ldots,n$, we set  $$b_i=\rho(i)\mbox{\qquad and\qquad } a_i=d-\rho([n]\setminus \{i\}),$$
 and set $$\beta_i=\rho([i])\mbox{\qquad  and\qquad } \alpha_i=d-\rho([n]\setminus [i]).$$

 It follows that $\Pb$ is a PLP-polymatroid of type$(\ab,\bb|{\bm \alpha},{\bm \beta}).$
\end{proof}

\begin{Proposition} \label{4} Let $\Pb$ be a discrete polymatroid on the ground set $[4]$ and suppose that $\rho(i)<\min\{\rho(i,3),\rho(i,4)\}$
 for $i=1,2$.
Then $\Pb$ satisfies the left-sided  strong symmetric exchange property if and only if it is a PLP-polymatroid.
\end{Proposition}

\begin{proof}
If $\rho(i)=0$ for some $i$, $\Pb$ is isomorphic to a discrete polymatroid on the ground set $[n]$ with $n\leq 3$. Therefore $\Pb$  is of Veronese type (see \cite[Example 2.6]{HH}) and there is nothing to prove. Hence we may  assume that $\rho(i)>0$ for $i=1,\ldots,4$.

The proof of the  ``if" part follows from Proposition~\ref{pro}.

Proof of the  ``only if" part: Set $d=\rho([4])$. Suppose that for some $i\in [4]$, say for $i=1$, such that $c=\rho([4]\setminus \{i\})<d$. Then $u_4\geq d-c$ for any $\ub\in B$. Let $B'=\{\ub-(d-c){\bm \epsilon}_4\:\;\ub\in B\}$. It follows that $B'$ is the set of bases of a discrete polymatroid $\Pb'$ whose ground set rank function $\rho'$ satisfies $\rho'([3])=\rho'([4])$. Note that $\Pb'$ satisfies the left-sided  strong symmetric exchange property if and only if $\Pb$ satisfies the same property and that $\Pb'$ is a PLP-polymatroid
if and only if $\Pb$ is a PLP-polymatroid. Hence we assume that $\rho([4]\setminus \{i\})=\rho([4])=d$ for all $i$ from the beginning.

\medskip
Assume that $\Pb$ is not a pruned  lattice path polymatroid. Then,  by Proposition~\ref{cri},
there is one pair  $\{i,j\}\in \{\{2,3\},\{1,4\},\{1,3\},\{2,4\}\}$ such that $\{i,j\}$ is $\rho$-closed and $\rho$-inseparable, i.e., $\rho(i,j)<\rho(i)+\rho(j)$ and $\rho(i,j)<d$.

If  $\{i,j\}=\{1,4\}$, then, in view of  Equation (\ref{rho}), we see that the  vectors $\ub=(\rho(1,4)-\rho(4),d-\rho(1,4),0,\rho(4))$ and   $\vb=(\rho(1,4)-\rho(4)+1,d-\rho(1,4)-1,1,\rho(4)-1)$ belong to $B$.   Now by the  left-sided strong symmetric exchange property of $B$, it follows that the vector $\ub'=(\rho(1,4)-\rho(4)+1,d-\rho(1,4)-1,0,\rho(4))$ belongs to $B$, a contradiction, since  $\ub'(1)+\ub'(4)=\rho(1,4)+1$.

If  $\{i,j\}=\{1,3\}$, consider the vectors $\ub=(\rho(1,3)-\rho(3),d-\rho(1,3),\rho(3),0)$ and $\vb=(\rho(1,3)-\rho(3)+1,d-\rho(1,3)-1,\rho(3)-1,1)$;

If $\{i,j\}=\{2,3\}$, consider the vectors $\ub=(d-\rho(2,3),\rho(2,3)-\rho(3),\rho(3),0)$ and $\vb=(d-\rho(2,3)-1,\rho(2,3)-\rho(3)+1,\rho(3)-1,0)$;

And, if $\{i,j\}=\{2,4\}$, consider the vectors $\ub=(d-\rho(2,4),\rho(2,4)-\rho(4),0,\rho(4)$ and $\vb=(d-\rho(2,4)-1,\rho(2,4)-\rho(4)+1,1,\rho(4)-1)$.

In each of  these cases, we obtain a similar  contradiction as in the case when $\{i,j\}=\{1,4\}$. This  completes the proof.
\end{proof}

In view of Example~\ref{2.4}, the conditions that $\rho(i)<\min\{\rho(i,3),\rho(i,4)\}$
 for $i=1,2$ in Proposition~\ref{4} cannot be skipped.

\medskip
Next, we discuss another case where Conjecture~\ref{C1} holds.
   Let $\Pb_1,\ldots,\Pb_k$ be discrete polymatroids on the ground set $[n]$, and let $B_i$ be the set of bases of $\Pb_i$ for $i=1,\ldots,k$. Then $\Pb_1\bigvee \ldots\bigvee \Pb_k=\{\ub_1+\cdots+\ub_k\:\; \ub_i\in \Pb_i \mbox{\ for\  } i=1,\ldots,k\}$ is a discrete polymatroid and the set of its bases is $B_1+\cdots+B_k$. Moreover, the polymatroidal ideal of $\Pb_1\bigvee \ldots\bigvee \Pb_k$ is the product of the polymatroidal ideals of $\Pb_1,\ldots,\Pb_k.$

\begin{Lemma} \label{B+C} Let $\Pb_1$ be the discrete polymatroid whose  polymatroidal ideal is the transversal ideal $P_{[a,b]}P_{[c,d]}$ with $c<a<b<d$. Then for any discrete polymatroid $\Pb_2$, the polymatroidal sum $\Pb_1\bigvee\Pb_2$
 satisfies  neither the right- nor the left-sided strong symmetric  exchange property.
\end{Lemma}

\begin{proof} Let $\Pb=\Pb_1\bigvee\Pb_2$, and let $\rho_1$, $\rho_2$, $\rho$ be the ground set rank functions of $\Pb_1,\Pb_2,\Pb$ respectively.  Then $\rho=\rho_1+\rho_2$. By \cite[Lemma 3.2]{HH}, there exists a base ${\bold w}$ of $\Pb_1$ such that $w_c+w_d=\rho_1(c,d)$. Let $\ub={\bm \epsilon}_c+{\bm \epsilon}_b$ and $\vb={\bm \epsilon}_a+{\bm \epsilon}_d$. Then both $\ub$ and $\vb$ are bases of  $\Pb_2$.  Assume that $\Pb$ satisfies the right-sided strong symmetric exchange property. Since $({\bold w}+\ub)_b>({\bold w}+\vb)_b$ and $\sum_{i>b} ({\bold w}+\ub)_i<\sum_{i>b}({\bold w}+\vb)_i$, the right-sided strong symmetric exchange property of $\Pb$ implies that ${\bold w}+{\ub}-{\bm \epsilon}_b+{\bm \epsilon}_d={\bold w}+{\bm \epsilon}_c+{\bm \epsilon}_d$ belongs to $ B+C$. Hence $\rho(c,d)\geq \rho_1(c,d)+2$, which is impossible, since  $\rho_2(c,d)=1$. Similarly, $\Pb$ does not satisfy the left-sided strong symmetric exchange property.
\end{proof}

\begin{Proposition} Let $\Pb$ be the discrete  polymatroid whose polymatroidal ideal $I$ is  $P_{[s_1,t_1]}\cdots P_{[s_d,t_d]}$. Then the following conditions are equivalent:

\begin{enumerate}
\item[(a)] $\Pb$ satisfies the one-sided
strong symmetric exchange property;

\item[(b)]  after a suitable rearrangement of  the factors of $I$, we have  $$s_1\leq s_2\leq \cdots \leq s_d\quad\text{and}\quad t_1\leq t_2\leq \cdots \leq t_d;$$

\item[(c)] $\Pb$ is a lattice path polymatroid.
\end{enumerate}
\end{Proposition}

\begin{proof} $(a)\Rightarrow (b)$: We can rearrange the order of prime ideals in the expression of $I=P_{[s_1,t_1]}\cdots P_{[s_d,t_d]}$ in this way: if $s_i<s_j$, then $P_{[s_i,t_i]}$ is placed before $P_{[s_j,t_j]}$; if $s_i=s_j$ and $t_i\leq t_j$, then $P_{[s_i,t_i]}$
is placed before $P_{[s_j,t_j]}$. After this rearrangement, we have $s_1\leq s_2\leq \cdots \leq s_d$, and $t_i\leq t_j$ for any  $i<j$ with $s_i=s_j$. It suffices to prove that $t_i\leq t_j$ if $s_i<s_j$.  But this follows from Lemma~\ref{B+C}.

$(b)\Rightarrow (c)$  follows from  Lemma~\ref{trans1} and
$(c)\Rightarrow (a)$   follows from  Proposition~\ref{pro}.
\end{proof}

\medskip
In the following proposition we will show that any power of a PLP-polymatroidal ideal is again a PLP-polymatroidal ideal.

\begin{Proposition}\label{power} Let $I$ be the PLP-polymatroidal ideal of type $(\ab,\bb|{\bm \alpha},{\bm \beta})$. Then  $I^k$ is the PLP-polymatroidal ideal of type $(k\ab,k\bb|k{\bm \alpha},k{\bm  \beta})$ for any $k>1$.
\end{Proposition}

\begin{proof} Let $J$ be the PLP-polymatroidal ideal of type$(k\ab,k\bb|k{\bm \alpha},k{\bm \beta})$. Then  $I^k\subseteq J$.  Let $\xb^{\ub}=x_1^{u_1}\ldots x_n^{u_n}$ be a minimal generator in $J$. We will show that  there are minimal generators $\xb^{\vb_1},\ldots, \xb^{\vb_k}$ of $I$ such that ${\xb}^{\ub}=\xb^{\vb_1}\ldots {\xb}^{\vb_k}$, that is, $\ub=\vb_1+\cdots+\vb_k$.

For each $i$, there exist $s_i,t_i\in \ZZ_+$ such that $u_1+\cdots+u_i=ks_i+t_i$ and  $0\leq t_i<k$. For  $j=1,\ldots,k$ we define $\vb_j$ by $$\vb_j(1)+\cdots+\vb_j(i)=s_i+1, \quad \text{for} \quad  1\leq j\leq t_i, \quad 1\leq i\leq n,$$  and $$\vb_j(1)+\cdots+\vb_j(i)=s_i, \quad \text{for} \quad  t_i+1\leq j\leq k, \quad 1\leq i\leq n.$$

To show that $\vb_j$ belongs to $\ZZ_+^n$ for all $j\in [k]$, we only need  to show that $\vb_j(1)+\cdots+\vb_j(i+1)\geq \vb_j(1)+\cdots+\vb_j(i)$ for all  $1\leq i\leq n-1$ and $j\in [k]$. Since $s_{i+1}\geq s_i$, we can assume that $\vb_j(1)+\cdots+\vb_j(i+1)=s_{i+1}$ and $\vb_j(1)+\cdots+\vb_j(i)=s_i+1$.  Note that in this case,  $j>t_{i+1}$ and $j\leq t_i$. This implies  $t_{i+1}<t_i$.  Since $ks_{i+1}+t_{i+1}\geq ks_i+t_i$, it follows that $ks_{i+1}>ks_i$ and so $s_{i+1}\geq s_i+1$, as required.
\vspace{2mm}

Since $$\sum _{1\leq j\leq k}\vb_j(1)+\cdots+\sum_{1\leq j\leq k}\vb_j(i)=ks_i+t_i, \quad i=1,\ldots,n,$$ we have $\ub=\vb_1+\cdots+\vb_k$.
\vspace{2mm}

We have $\alpha_i\leq \vb_j(1)+\cdots+\vb_j(i)\leq \beta_i$  for  $1\leq j\leq k$  and $1\leq i\leq n$, since $k\alpha_i\leq ks_i+t_i\leq k\beta_i$.
\vspace{2mm}

It remains to be shown that  $a_i\leq \vb_j(i)\leq b_i$ for $1\leq j\leq k$ and $1\leq i\leq n.$
If $t_i>t_{i-1}$, then $\vb_j(i)$ is either  $s_i+1-s_{i-1}$ or $s_i-s_{i-1}$. Note that $u_i=k(s_i-s_{i-1})+t_i-t_{i-1}$. Therefore we have $a_i\leq s_i-s_{i-1}+\frac{1}{k}(t_i-t_{i-1})\leq b_i$, and hence $a_i\leq \vb_j(i)\leq b_i$ for  $1\leq j\leq k$.
\vspace{2mm}

 Similarly one shows  that  $a_i\leq \vb_j(i)\leq b_i$   for $1\leq j\leq k$  when $t_i=t_{i-1}$ or $t_i<t_{i-1}$. This completes  our proof.
\end{proof}

 We do not know if  the product of PLP-polymatroidal ideals is always a PLP-polymatroidal ideal. However there is an  example to show the product of LP-polymatroidal ideals may be not a LP-polymatroidal ideal.

\begin{Example} \em Set $I_1=(x_1x_2^2x_3^2, x_1x_2^3x_3, x_2^3x_3^2,x_2^4x_3 )$ and set $I_2=(x_1,x_2,x_3)$. Then  $I_1$ and $I_2$ are LP-polymatroidal ideals of type$((0,3,5),(1,4,5))$ and $((0,0,1),(1,1,1))$ respectively. We claim that $I_1I_2$ is not a  LP-polymatroidal ideal. If it is, there are $\alpha_i,\beta_i,i=1,2$ such that $I_1I_2$ is generated by $\xb^{\ub}$ with $\ub$ satisfying $\alpha_1\leq u_1\leq \beta_1, \alpha_2\leq u_1+u_2\leq \beta_2, u_1+u_2+u_3=6$. One have $a_1=0,b_1\geq 2$ and $a_2\leq 3,b_2\geq 5$. This implies $x_1^2x_2x_3^3\in I_1I_2$,  a contradiction.

\end{Example}

\vspace{4mm}

\section{Gorensteinness  of base rings of a special type of PLP-polymatroid}

Let $\Pb$ be a PLP-polymatroid for which the set $B$ of its bases consists of  vectors $\ub\in \ZZ_+^n$ satisfying:
$$ 0\leq u_i\leq b_i \mbox{\qquad\qquad \qquad for\qquad\qquad }i=1,\ldots, k,$$
$$0\leq u_1+\cdots+u_i\leq \beta_i \mbox{\qquad for\qquad }i=k+1,\ldots, n-1,$$
and $$u_1+\cdots+u_n=d.$$ Here $b_i>0$ for $i=1,\ldots,k$ and $1\leq \beta_{k+1}\leq\ldots\leq\beta_{n-1}\leq d$.
We call this special type of  PLP-polymatroids to be {\it SPLP-polymatroids}.
In this section, we will classify SPLP-polymatroids whose base rings are Gorenstein. We do this in two steps. First we identity the base ring of a SPLP-polymatroid   with the Ehrhart ring of a certain integral polymatroid. Then our result follows by applying \cite[Theorem 7.3]{HH}, which describes perfectly  the integral  polymatroids whose Ehrhart rings are Gorenstein.

\medskip

 In general, if $\MP$ is an integral  convex  polytope contained in $\RR_+^n$, then the {\it Ehrhart ring}  $K[\MP]$ is defined to be the $K$-subalgeba of $K[t_1,\ldots, t_n,s]$ generated by monomials ${\bold t}^{\ub}s^i$ with $\ub\in i\MP\cap \ZZ_+^n$. Assume further that $\MP$ is a polymatroid and $\rho$ is the ground set rank function of $\MP$.  Then, \cite[Theorem 7.3]{HH} says that $K[\MP]$ is Gorenstein if and only if there exists an integer $\delta\in \ZZ_+$ such that $\rho(A)=\frac{1}{\delta}(|A|+1)$ for all $\rho$-closed and $\rho$-inseparable subset $A$ of $[n]$.

  Let $\Pb$ be the discrete polymatroid $\MP\cap \ZZ_+^n$. Then, since $\MP$ has the  integer decomposition property, the Ehrhart ring  $K[\MP]$ is isomorphic to $K[\Pb]$,  which by definition is  the $K$-subalgeba of $K[t_1,\ldots, t_n,s]$ generated by monomials ${\bold t}^{\ub}s$ with $\ub\in \Pb$, see \cite{HH}.  Recall that an integral convex polytope $\MP$ is said to have the {\em integer decomposition property} provided  that for any integer $q\geq 1$ and any $\wb\in \ZZ_+^n$ which belongs to $q\MP=\{q\vb\:\;\vb\in \MP \}$, there exist $\ub_1,\ldots,\ub_q\in \MP\cap \ZZ_+^n$ such that $\wb=\ub_1+\cdots+\ub_q$.

\begin{Proposition}  Let $\Pb$ be a SPLP-polymatroid as given  above. Then the following statements are equivalent:

 {\em (1)}  The base ring $K[B(\Pb)]$ is Gorenstein;

  {\em (2)}  For any $k+1\leq i\leq  n-2$ with $\beta_i<\beta_{i+1}$ and for any $1\leq j\leq k$ with $b_j<\beta_{k+1}$, one has: $$\frac{i+1}{\beta_i}=\frac{2}{b_j}=\frac{n}{\beta_{n-1}}$$ is a positive integer.
\end{Proposition}

\begin{proof}
Let $\Pb'$ be the set of integer  vectors $\ub=(u_1,\ldots,u_{n-1})\in \ZZ_+^{n-1}$ subject to
the first $(n-1)$ linear inequalities in the definition of a SPLP-polymatroid. Then $\Pb'$ is a discrete polymatroid on the ground set $[n-1]$. We claim that $K[B(\Pb)]\cong K[\Pb']$. In fact, $K[B(\Pb)]=K[\tb^{\ub}s^{d-|\ub|}\:\; \ub\in \Pb']$ and $K[\Pb']=K[\tb^{\ub}s\:\;\ub\in \Pb']$. Since for any coefficients   $k_{\ub}\in K$ with $\ub\in \Pb'$, $\sum_{\ub\in \Pb'}k_{\ub}(\ub, d-|\ub|)=0$ if and only if  $\sum_{\ub\in \Pb}k_{\ub}(\ub, 1)=0$, it follows that $K[B(\Pb)]$ and $K[\Pb']$ have the same relation lattices and so $K[B(\Pb)]\cong K[\Pb']$, as claimed.
\medskip

Let $\rho'$ be the ground set rank function of $\Pb'$. To determine the $\rho'$-closed and $\rho'$-inseparable subsets of $\Pb'$, we use the uniqueness of the irredudant decomposition  of $\Pb'$, see Equation (\ref{irr}).  First by the definition of $\Pb'$, we have the following decomposition:  \begin{eqnarray} \label{dec}\qquad \Pb'= (\bigcap_{i=1}^k\{\ub\in \ZZ^{n-1}_+\:\;u_i\leq b_i \})\cap (\bigcap_{i=k+1}^{n-1}\{\ub\in \ZZ_+^{n-1}\:\;u_1+\cdots+u_i\leq \beta_i\}). \end{eqnarray}

In general, if $A=\bigcap_{i\in I} A_i$, where $I$ is a finite index set, then $A_i$ is called {\em superfluous} in this decomposition if $A_i\supseteq \bigcap_{j\neq i}A_j$. We can  omit all superfluous terms step by step to achieve an irredundant decomposition of $A$.

\vspace{1mm}
We make the following observations:

\vspace{1mm}
(a) For $1\leq i\leq k$, the term $\{\ub\in \ZZ_+^{n-1}\:\; u_i\leq b_i\}$ is superfluous  in (\ref{dec})  if and only if $b_i\geq \beta_k$,

\vspace{1mm}
(b) For $k+1\leq i\leq n-2$, the term  $\{\ub\in \ZZ_+^{n-1}\:\; u_1+\cdots+u_i\leq \beta_i\}$ is superfluous  in (\ref{dec}) if and only if $\beta_i= \beta_{i+1}$,

\vspace{1mm}

(c) the term $\{\ub\in \ZZ_+^{n-1}\:\; u_1+\cdots+u_{n-1}\leq \beta_{n-1}\}$ is not superfluous in (\ref{dec}).
\vspace{1mm}

We only prove (b) since the proofs of the other facts are similar.

\vspace{2mm}
If $\beta_i=\beta_{i+1}$, then the term $\{\ub\in \ZZ_+^{n-1}\:\; u_1+\cdots+u_i\leq \beta_i\}$ is superfluous since it contains $\{\ub\in \ZZ_+^{n-1}\:\; u_1+\cdots+u_{i+1}\leq \beta_{i+1}\}$.
\vspace{2mm}

If $\beta_i<\beta_{i+1}$, we define  $\ub\in\ZZ_+^{n-1}$ by $u_i=\beta_{i+1}$ and $u_j=0$ for $j\neq i$. Then $\ub$ does not belong to $\Pb'$ since $u_1+\cdots+u_i>\beta_i$, but it belongs to the decomposition obtained from (\ref{dec}) by dropping  the term $\{\ub\in \ZZ^n\:\; u_1+\cdots+u_i\leq \beta_i\}$. Hence $\{\ub\in \ZZ^n\:\; u_1+\cdots+u_i\leq \beta_i\}$ is not superfluous. This proves (b).

\vspace{2mm}
Note that if two terms are superfluous in (\ref{dec}), then one  term is still superfluous in the decomposition obtained from (\ref{dec}) by dropping the other term. (This is not true for an arbitrary decomposition $A=\bigcap_{i\in I} A_i$). Hence we achieve an irredundant decomposition of $\Pb'$ by dropping all superfluous terms in (\ref{dec}). It follows from Equation (\ref{irr}) that  there are three classes of $\rho'$-closed and $\rho'$-inseparable subsets of $\Pb'$:

 (a) the subsets  $\{i\}$, where $1\leq i\leq k$ and $b_i<\beta_k$,

  (b) the subsets $[i]$, where $k+1\leq i\leq n-2$ and $\beta_i<\beta_{i+1}$,

   (c) $[n-1]$. Now a direct application of
\cite[Theorem 7.3]{HH}  yields  our result.
\end{proof}

\section{Linear quotients for PLP-polymatroidal ideals}

From this section on, we will turn to investigate the algebraic properties of  polymatriodal ideals of some classes of  PLP-polymatoids.

\medskip
 In view of  \cite[Theorem 12.6.2 and Theorem 12.7.2]{HHBook} and their proofs, we see that a polymatroidal ideal $I$ has linear quotients if its minimal generators are arranged in either the lexicographical
order or the reverse  lexicographical
order. In this section we will show that if $I$ is a PLP-polymatroidal ideal, then its linear  quotients are more easily trackable. We will use this result repeatedly  in the following sections.

\begin{Remark} \em
 Let $I$ be a PLP-polymatroidal ideal of type  $(\ab,\bb|{\bm \alpha},{\bm \beta})$. Then $I$ is isomorphic to a PLP-polymatroidal ideal of type $(\mathbf{0},\bb_1|{\bm \alpha_1},{\bm \beta_1})$ for suitable vectors $\bb_1,{\bm \alpha_1}$ and ${\bm \beta_1}$. In fact, let $J$ be the ideal generated by monomials $x_1^{u_1}\cdots x_n^{u_n}$ satisfying
$$0 \leq u_i\leq b_i-a_i, \forall i=1,\ldots,n$$
and $$\alpha_{i}-\sum_{j=1}^i a_i \leq u_1+\cdots+u_i\leq \beta_{i}-\sum_{j=1}^i a_i, \forall i=1,\ldots,n. $$
Then,  since $I=x_1^{a_1}\cdots x_n^{a_n}J$, the ideals $I$ and $J$ are isomorphic as modules and hence $I$ and $J$ have the same projective dimension and the same depth. In what follows we always assume that $I$ is a PLP-polymatroidal ideal of type $(\mathbf{0},\bb|{\bm \alpha},{\bm \beta})$.
\end{Remark}

  \begin{Lemma}\label{quo} Let $I$ be a PLP-polymatroidal ideal of type $(\mathbf{0},\bb|{\bm \alpha},{\bm \beta})$. Let $G(I)=\{m_1,\ldots,m_r\}$ be the set of minimal generators of $I$ such that $m_1>m_2>\cdots>m_r$ with respect to the lexicographical order. Fix $1\leq q\leq r$ and let $J=(m_1,\ldots,m_{q-1})$.  Write  $m_q=x_1^{t_1}\cdots x_n^{t_n}$. Then

\vspace{2mm}
 { \em{(a)}} $x_n$ is not in $J:m_q$;

\vspace{2mm}
  {\em{(b)}} for any $1\leq i\leq n-1$,   $x_i\in J:m_q$ if and only if $t_i< b_i$ and  $t_1+\cdots+t_{i}<\beta_i$.
  \end{Lemma}

\begin{proof}  If $x_i\in J:m_q$, then there is $l<q$ with $x_i=m_l/[m_l,m_q]$. Here $[u,v]$ denotes the greatest  common divisor of  of the monomials $u$ and $v$. Since $m_l>m_q$, we have $m_l=x_1^{t_1}\cdots x_i^{t_i+1}\cdots x_j^{t_j-1}\cdots  x_n^{t_n}$ for some  $j$ with $j>i$.   Hence $i<n$ and  $x_n\notin J:m_q$.  This proves (a).

\vspace{2mm}
Proof of (b): Let  $i<n$  and  $x_i\in J:m_q$. Then as in the proof of  (a) there exists $m_l\in G(I)$ with $m_l=x_1^{t_1}\cdots x_i^{t_i+1}\cdots x_j^{t_j-1}\cdots  x_n^{t_n}$ for some  $j$ with $j>i$.  It follows  that $t_i< t_i+1\leq b_i$ and $t_1+\cdots+t_i<t_1+\cdots+t_i+1\leq \beta_i$.

Conversely let $i<n$ and assume that $t_i<  b_i$ and $t_1+\cdots+t_i< \beta_i$. Let $j$ be the smallest integer $k$ with the property that  $k>i$ and $t_1+\cdots+t_k=\beta_k$. Note that such a $k$ exists since $t_1+\cdots+t_n=\beta_n$. By the choice of $j$,  we have $t_j>\beta_j-\beta_{j-1}\geq 0$. Set $m=x_1^{t_1}\cdots x_i^{t_i+1}\cdots x_j^{t_j-1}\cdots  x_n^{t_n}$. Then  $m$ belongs to $G(I)$ with $m>m_q$ and
$m/[m,m_q]=x_i$. Hence $x_i\in J:m_q$, and this completes the proof of $(b)$.
\end{proof}

\medskip
For any $\xb^{\ub}=x_1^{u_1}\cdots x_n^{u_n}\in G(I)$, we set
\begin{eqnarray}
\label{nu}
N(\ub)= |\{1\leq i\leq n-1\:\;u_i<b_i,u_1+\cdots+u_i<\beta_i\}|.
\end{eqnarray}

As an immediate  consequence of Lemma~\ref{quo} together with \cite[Corollary 8.2.2]{HHBook}, we obtain:

\begin{Proposition} \label{nu1}  Let $I$ be a PLP-polymatroidal ideal of type $(\mathbf{0},\bb|{\bm \alpha},{\bm \beta})$. Then $$\depth S/I=n-1-\max\{N(\ub)\:\; \xb^{\ub}\in G(I)\}.$$

\end{Proposition}

\medskip
In the rest part of this section we will decide the associated prime ideals   of  LP-polymatroidal ideals. we begin with:

\begin{Proposition} \label{simple}
Let $I$ be a LP-polymatroidal ideal of type $({\bm \alpha},{\bm \beta})$. Then the following statements are equivalent:
\begin{enumerate}
\item[ (a)]  $\mathrm{depth}\ S/I=0$;

\item[(b)]  $\mathrm{m}\in \mathrm{Ass}(S/I)$;

\item[(c)] $\alpha_i<\beta_i$ for $i=1,\ldots,n-1$.
\end{enumerate}Here $\mathrm{m}$ denotes the maximal homogeneous ideal $(x_1,\ldots,x_n)$ of $S=K[x_1,\ldots,x_n]$
\end{Proposition}

\begin{proof} The equivalence between  (a) and (b) is well-known and it holds for all monomial ideals.

 (a)$\Leftrightarrow$  (c): Note that in this special case, $$N(\ub)=|\{1\leq i\leq n-1\:\; u_1+\cdots+u_i<\beta_i\}|$$ for any $\ub$ with $\xb^{\ub}\in G(I).$ Let $B$ denote the set of bases of our discrete polymatroid.

 If $\alpha_i<\beta_i$ for $i=1,\ldots,n-1$, then $\vb:=(\alpha_1,\alpha_2-\alpha_1,\ldots,\alpha_n-\alpha_{n-1})$ belongs to $B$ and $N(\vb)=n-1$.   Hence $\mathrm{depth}\ S/I=0$ by Proposition~\ref{nu1}. Conversely, if  $\mathrm{depth}\ S/I=0$, then there exists $\ub\in B$ such that $N(\ub)=n-1$ by Proposition~\ref{nu1} again. This implies $u_1+\cdots+u_i<\beta_i$ for $i=1,\ldots,n-1$. Since $\alpha_i\leq u_1+\cdots+u_i$  for $i=1,\ldots,n-1$ by definition, the result follows.\end{proof}

\begin{Lemma} \label{form} Let $I$ be a LP-polymatroidal ideal. Then every associated prime ideal of $I$ is of the form $P_{[s,t]}$ for some $s\leq t$.
\end{Lemma}
\begin{proof} The assertion follows from Lemma~\ref{trans1} together with \cite[Theorem 3.7]{HHV}.
\end{proof}

 Following \cite{HRV}, we denote by $S(P)$ the polynomial ring in the variables belonging to $P$, and $I(P)$ the {\it monomial localization} of $I$ at $P$. Recall that $I(P)$ is the monomial ideal of $S(P)$ obtained from $I$ as the image of the map $\phi: S\rightarrow S(P)$ defined by $\phi(x_i)=x_i$ if $x_i\in P$ and $\phi(x_i)=1$, otherwise. Let $A\subseteq [n]$. We denote $I(P_{A})$ by $I(A)$ and $S(P_A)$ by $S(A)$ for short. One can check that if $I,J$ are monomial ideals then $(IJ)(A)=I(A)J(A)$, and if $A\subseteq B$ then $I(A)= I(B)(A)$. In the  later proofs  we need the following  facts.

\begin{Proposition} \label{basic} {\em{(a)}} Let $I$ be a monomial ideal. Then $P\in \mathrm{Ass}(S/I)$ if and only if $\mathrm{depth}\ S(P)/I(P)=0$.

{\em{(b)}} Let $I$ be a polymatroidal ideal and fix $i\in [n]$. Set $$a_i=\max\{u_i\:\; \xb^{\ub}\in G(I)\}.$$ Then $I([n]\setminus \{i\})$ is again a polymatroidal ideal. Moreover it is  generated by monomials $\xb^{\ub}/x^{a_i}$, where  $\xb^{\ub}\in G(I)$ with $u_i=a_i$.

\end{Proposition}

\begin{proof} (a) This has been observed in \cite[Lemma 2.11]{FHT}.

(b) This can be seen  from \cite[Proposition 2.1]{HRV} and its proof.
\end{proof}

Let $I$ be a LP-polymatroidal ideal of type $({\bm \alpha},{\bm \beta})$. By  Proposition~\ref{basic}(b) we have
$I([s,n])$ is generated by $\xb^{\ub}$ with $\ub$ satisfying:
 $$\max\{\alpha_{i}-\beta_{s-1},0\}\leq u_s+\cdots+u_i\leq \beta_{i}-\beta_{s-1}\qquad  \text{\ for\ } i=s,\ldots,n,$$
and $I([1,t])$ is generated by $\xb^{\ub}$ with $\ub$ satisfying:
$$ \alpha_i\leq u_1+\cdots+u_i\leq \min\{\alpha_t,\beta_i\}\qquad \text{ for } i=1,\ldots,t. $$
Hence $I([s,t])$ is an ideal in $S([s,t])$ generated by ${\xb}^{\ub}$ with $\ub$ satisfying:
$$\max\{\alpha_i-\beta_{s-1},0\}\leq u_s+\cdots+u_i\leq \min\{\max\{\alpha_t-\beta_{s-1},0\},\beta_i-\beta_{s-1}\}$$ for $i=s,\ldots,t$.

\begin{Proposition} \label{lattice} Let $I$ be a LP-polymatroidal ideal of type $({\bm \alpha},{\bm \beta})$. Then $P_{[s,t]}$ belongs to $\mathrm{Ass}(S/I)$ if and only if
 \begin{center}   $\beta_{s-1}<\beta_s$, $\beta_{s-1}<\alpha_t$, $\alpha_{t-1}<\alpha_t$,\quad  and\quad $\alpha_i< \beta_i$ for $i=s,\ldots,t-1$.
  \end{center}In particular, $(x_t)\in \mathrm{Ass}(S/I)$ if and only if $\alpha_t>\beta_{t-1}$.  Here we use the convention that $\beta_0=0$.
\end{Proposition}

\begin{proof} In view of Lemma~\ref{basic}(a) and Proposition~\ref{simple}, we see that $P_{[s,t]}$ belongs to $\mathrm{Ass}(S/I)$ if and only if $\max\{\alpha_i-\beta_{s-1},0\}<\max\{\alpha_t-\beta_{s-1},0\}$ and $\max\{\alpha_i-\beta_{s-1},0\}<\beta_i-\beta_{s-1}$ for $i=s,\ldots,t-1$, which is equivalent to the conditions given in Proposition~\ref{lattice}.
\end{proof}

\section{The depth and the associated prime ideals of left PLP-polymatroidal ideals}

A PLP-polymatroidal   ideal $I$ in $S=K[x_1,\ldots,x_n]$ is called a {\em left PLP-polymatroidal   ideal}, if $I$ is generated  by monomials $\xb^{\ub}$ such that
 $$0\leq u_i\leq b_i \text{\qquad for\ } 1\leq i\leq k,$$ and
  $$ \alpha_{k+i}\leq u_1+\cdots+u_{k+i}\leq \beta_{k+i}  \text {\qquad for } i=1,\ldots,n-k,$$  where $1\leq k\leq n-1$ and  $b_i>0$, $\alpha_{k+1}\leq \cdots\leq  \alpha_{n}=d$, $\beta_{k+1}\leq \cdots \leq \beta_n=d$.   Hence a left PLP-polymatroidal ideal is a PLP-polymatroidal ideal of type $(\mathbf{0},\bb| {\bm \alpha},{\bm \beta})$ such that there exists $k\in [2,n-1]$ such that $b_i\geq d$ for all $i\geq k+1$ and $\alpha_i=0$, $\beta_i=\beta_{k+1}$ for all $i\leq k$.

  In this section we will investigate the depth and the associated prime ideals of this class of ideals and their powers.

\medskip
 Let $\xb^{\ub}=x_1^{u_1}\cdots x_n^{u_n}$. The support of $\ub$, denoted by $\mathrm{supp}(\ub)$, is the set $\{i|\; u_i\neq 0\}$. If $I$ is a monomial ideal, then the support of $I$, denoted by $\mathrm{supp}(I)$,  is the set $\bigcup_{{\xb}^{\ub}\in G(I)}\mathrm{supp}(\ub)$.

\begin{Lemma}  \label{H} Let $I_1,I_2,\ldots,I_t\subseteq S$ be nonzero monomial ideals whose supports are pairwise disjoint. For $j=1,\ldots, t$, let  $S_j$ be a polynomial ring whose set of  variables $V_j$ contains the set $\{x_i |\; i\in \supp(I_j)\}$. We further assume that $V_1,\ldots,V_t$ are pairwise disjoint and  that $V=V_1\union V_2\union \ldots \union V_t$ is the set of variables of $S$.  For $j=1,\ldots,t$, let $L_j= S_j\sect I_j$. Then

\begin{enumerate}
\item[(a)] $\depth S/(I_1\cdots I_t)=t-1+\sum_{j=1}^t \depth S_j/L_j$;
\item[(b)] $\Ass(S/(I_1\cdots I_t))=\bigcup _{1\leq j\leq t} \Ass(S/I_j).$
\end{enumerate}
\end{Lemma}

\begin{proof}By induction, we only need to consider the case when $t=2$.

(a) Denote by $p_i$ the projective dimension of $S/I_i$ for $i=1,2$. Then $$\projdim_S\ S/(I_1+I_2)=p_1+p_2.$$  This formula follows from the fact that the tensor product of minimal free resolutions of $S/I_1$ and  $S/I_2$ is a minimal free resolution of $S/(I_1+I_2)$, see e.g. \cite[Corollary 2.2]{JK}.

We will show that $\projdim_S\ S/(I_1I_2)=p_1+p_2-1.$ For this we consider the short exact sequence: $$0\rightarrow S/(I_1I_2)\rightarrow S/I_1\oplus S/I_2\rightarrow S/(I_1+I_2)\rightarrow 0.$$
It induces the following long exact sequence:
$$\cdots\rightarrow \mathrm{Tor}_{j+1}^S(S/I_1\oplus S/I_2, K)\rightarrow \mathrm{Tor}_{j+1}^S(S/(I_1+ I_2), K)\rightarrow \mathrm{Tor}_j^S(S/(I_1I_2), K)$$$$\rightarrow \mathrm{Tor}_j^S(S/I_1\oplus S/I_2 ,K) \rightarrow \cdots$$

 Then, for $j> p_1+p_2-1$, since $\mathrm{Tor}_{j+1}^S(S/I_1\oplus S/I_2, K)\cong\mathrm{Tor}_j^S(S/I_1\oplus S/I_2 ,K)=0$, we have the isomorphism  $\mathrm{Tor}_j^S(S/(I_1I_2), K)\cong \mathrm{Tor}_{j+1}^S(S/(I_1+ I_2), K)=0$, and for $j=p_1+p_2-1$, since $\mathrm{Tor}_{j+1}^S(S/I_1\oplus S/I_2, K)=0$ and
 $\mathrm{Tor}_{j+1}^S(S/(I_1+ I_2), K)\neq 0$, we have  $\mathrm{Tor}_j^S(S/(I_1I_2), K)\neq 0$. It follows that $\projdim_S\ S/(I_1I_2)=p_1+p_2-1.$

It is not hard to see that $\projdim_S S/I_j=\projdim_{S_j} S_j/L_j$ for $j=1,\ldots,t$. Thus, by  the Auslander-Buchsbaum formula, one has
$$\depth S/(I_1I_2)=n-p_1-p_2+1=\sum_{i=1}^2(|V_i|-p_i)+1=\sum_{i=1}^2 \depth S_j/L_j+1.$$

\medskip

(b )  Considering the same short exact sequence as given in the proof of (a), it follows that
$$\mathrm{Ass}(S/(I_1I_2))\subseteq \mathrm{Ass}(S/I_1)\cup\mathrm{ Ass}(S/I_2)\subseteq \mathrm{Ass}(S/(I_1I_2))\cup \mathrm{Ass}(S/(I_1+I_2)).$$ It is not hard to see that  $\mathrm{Ass}(S/(I_1+I_2))=\{P_1+P_2\:\; P_i\in \mathrm{Ass}(S/I_i) \mbox{ for } i=1,2\}$, and that if $P_i\in \mathrm{Ass}(S/I_i)$ then $\mathrm{supp}(P_i)\subseteq \mathrm{supp}(I_i)$ for $i=1,2$. This implies that
$(\mathrm{Ass}(S/I_1)\cup \mathrm{Ass}(S/I_2))\sect \mathrm{Ass}(S/(I_1+I_2))=\emptyset$. Hence  $\mathrm{Ass}(S/I_1)\cup \mathrm{Ass}(S/I_2)=\mathrm{Ass}(S/(I_1I_2))$, as desired.
\end{proof}

In the remaining part of this section   $I$ always denotes the left PLP-polymatroidal ideal as given in the beginning of this section.

\begin{Proposition} \label{depthleft}  $\mathrm{depth}\ S/I=|\{k+1\leq i\leq n-1\:\;\alpha_i=\beta_i\}|$.
\end{Proposition}

\begin{proof} Set $t=|\{k+1\leq i\leq n-1\:\; \alpha_i=\beta_i\}|$. Define a vector $\vb$ by $v_i=0$ for $i=1,\ldots,k$, $v_{k+1}=\alpha_{k+1},$ and $v_i=\alpha_i-\alpha_{i-1}$ for $i=k+2,\ldots,n$. Then $\xb^\vb\in G(I)$   and $N(\vb)=n-1-t$ (see Equation (\ref{nu}) for the definition of $N(\ub)$). This implies $\depth S/I\leq t$ by Proposition~\ref{nu1}. Since $N(\ub)\leq n-1-t$ for any $\ub$ with $\xb^\ub\in G(I)$, we have $\depth S/I\geq t$, by using Proposition~\ref{nu1} again.
\end{proof}

Following \cite{HRV}, we use $\mathrm{dstab}(I)$ (resp.\ $\mathrm{astab}(I)$) for the smallest integer $k$ with the property that $$\depth\ S/I^k=\depth \ S/I^{\ell}\qquad  (\text{resp.}\quad \mathrm{Ass}(S/I^k)=\mathrm{Ass}(S/I^{\ell})) $$ for all $\ell\geq k$.

\begin{Corollary} \label{dstableft}
$\mathrm{dstab}(I)=1$.
\end{Corollary}

\begin{proof} The assertion follows from Proposition~\ref{depthleft} and Proposition~\ref{power}.
\end{proof}

Next, we will describe the associated prime ideals of the left PLP-polymatroidal ideal $I$.

\begin{Lemma} \label{ass} Let $P_A\in \mathrm{Ass}(S/I)$ and suppose $i\notin A$ for some $i\geq k+1$. Then  either $[1,i-1]\sect A=\emptyset$ or $[i+1,n]\sect A =\emptyset$.
\end{Lemma}

\begin{proof} Assume on the contrary that there exist $1\leq \ell<i<j\leq n$
such that  $\ell\in A$ and $j\in A$. Let $B=[1,n]\setminus \{i\}$. Note that  $A\subseteq B$. Let $a_i=\max\{u_i\:\;\xb^{\ub}\in G(I)\}$. There are two cases to consider.

  The case when $i=k+1$:  Since  $a_i=a_{k+1}=\beta_{k+1}$ and since $u_1=\cdots=u_k=0$  for any $\xb^\ub\in G(I)$ with $u_{k+1}=\beta_{k+1}$, we have  $\mathrm{supp} (I(B))\subseteq [k+1,n]$ by Lemma~\ref{basic}(b), and in particular,
$\mathrm{supp} (I(A))\subseteq [k+1,n]$. It follows that $\ell\notin \mathrm{supp} (I(A))$,  and so $\mathrm{depth}\ S(A)/I(A)\neq 0$. Hence  $P_A\notin \mathrm{Ass}(S/I)$, a contradiction.

  The case when $i>k+1$: Note that $a_i=\beta_i-\alpha_{i-1}$, and for any $\xb^\ub\in G(I)$ with $u_i=a_i$, we have $u_1+\cdots+u_{i-1}=\alpha_{i-1}$.
This implies $I(B)=JL$ by Lemma~\ref{basic}(b). Here $J$ is generated by the monomials  $x_1^{u_1}\cdots x_{i-1}^{u_{i-1}}$  such that the $u_1,\ldots,u_{i-1}$ satisfy $0\leq u_1\leq b_1,\ldots,0\leq u_k\leq b_k$, $\alpha_{k+1}\leq u_1+\cdots+u_{k+1}\leq \beta_{k+1}, \ldots, u_1+\cdots+u_{i-1}=\alpha_{i-1}$,  and $L$ is generated by the monomials $x_{i+1}^{u_{i+1}}\cdots x_n^{u_n}$  such that the $u_{i+1},\ldots,u_n$ satisfy $\max\{\alpha_{i+1}-\beta_{i},0\}\leq u_{i+1}\leq \beta_{i+1}-\beta_{i},\ldots, u_{i+1}+\cdots+u_n=d-\beta_i$. It follows that $I(A)=J(A)L(A)$. Since  $\ell\in A$ and $j\in A$,  $A$ contains $\mathrm{supp}(L(A))$ and $\mathrm{supp}(J(A))$  properly. By this fact and since $J(A)$ and $L(A)$ have disjoint supports, we have  $\mathrm{depth}\ S(A)/I(A) \neq 0$ by Lemma~\ref{H}(a) and so $P_A\notin \mathrm{Ass}(S/I)$, a contradiction again.
\end{proof}

\begin{Corollary} \label{possible}
 If $P_A\in \mathrm{Ass}(S/I)$, then  either $A$ is a  subinterval of $[k+2,n]$, or  $A=B\cup [k+1,t]$ for some $B\subseteq [1,k]$ and some $t\geq k+1$.
\end{Corollary}

\begin{proof} Let $P_A\in \mathrm{Ass}(S/I)$. In view of Lemma~\ref{ass}, we have $A\cap [k+1,n]$ is either empty or an interval. We claim that $A\cap [k+1,n]\neq \emptyset$. For this, let $B$ be a subset of $[1,k]$.   As  seen in the proof of Lemma~\ref{ass},  $\mathrm{supp}(I([1,n]\setminus \{k+1\}))\subseteq [k+2,n]$. It follows that $\mathrm{supp}(I(B))\subseteq [k+2,n]$, which implies particularly  $P_B\notin \mathrm{Ass}(S/I)$. This proves the claim.

Let $j$ be the smallest integer in $A\cap [k+1,n]$. If $j\geq k+2$, then $k+1\notin A$ and it follows that $A\cap [1,k]=\emptyset$ by Lemma~\ref{ass}, and so $A$ an interval of $[k+2,n]$.
If $j=k+1$, then $A\cap [k+1,n]=[k+1,t]$ for some $t\geq k+1$ and so $A=B\cup [k+1,t]$ and some $B\subseteq [1,k]$.
\end{proof}

\begin{Proposition} \label{ass1} Let $\emptyset \neq B\subseteq [1,k]$ and $t\in [k+1,n]$. Then   $P_{B\cup [k+1,t]}\in \mathrm{Ass}(S/I)$ if and only if
\[
\sum_{j\in \overline{B}}b_j <\min\{\alpha_t,\beta_{k+1}\},\quad\alpha_{t-1}<\alpha_t
\]
and
\[ \alpha_{i}<\beta_{i},\text{ for all $i=k+1,\ldots, t-1$.}
\]
Here $\overline{B}:=[k]\setminus B$ and $\alpha_k:=0.$
\end{Proposition}

\begin{proof} For any $t\geq k+1$, Lemma~\ref{basic}(b) implies that $I([1,t])$ is an ideal in $S([1,t])$ generated by $\xb^{\ub}$ with $\ub$ satisfying
$0\leq u_i\leq b_i$ for $i=1,\ldots,k$  and $\alpha_i\leq u_1+\cdots+u_{i}\leq \min\{\alpha_t,\beta_i\}$ for $i=k+1,\ldots,t$.
Let $B\subseteq [1,k]$. We claim that if $\sum_{j\in \overline{B}}b_j \geq \min\{\alpha_t,\beta_{k+1}\}$, then $P_{B\cup [k+1,t]}\notin \mathrm{Ass}(S/I)$. (Note that the latter is equivalent to the condition that $\depth S([1,t]\setminus  \overline{B})/I([1,t]\setminus  \overline{B})\neq 0$). For this, we use induction on $|\overline{B}|$.

If $|\overline{B}|=1$, say $\overline{B}=\{i_1\}$, then, since $\max\{u_{i_1}\:\; \mathbf{x}^\mathbf{u}\in I([1,t])\}=\min\{\alpha_t,\beta_{k+1}\}$, we have $k+1\notin \mathrm{supp}(I(B\cup [k+1,t]))$ by Lemma~\ref{basic}(b). In particular, $P_{B\cup [k+1,t]}\notin \mathrm{Ass}(S/I).$

 Suppose that $|\overline{B}|>1$. We assume there exists $i_1\in \overline{B}$ such that $b_{i_1}<\min\{\alpha_t,\beta_{k+1}\}$, because otherwise we can  argue as in the proceeding  paragraph and obtain that $P_{B\cup [k+1,t]}\notin \mathrm{Ass}(S/I)$. Then $I([1,t]\setminus \{i_1\})$ is an ideal in $S([1,t]\setminus \{i_1\})$ generated by $\xb^\ub$  with $\ub$ satisfying
$0\leq u_i\leq b_i$ for   $1\leq i\leq k$ with $i\neq i_1$, and $\max\{\alpha_i-b_{i_1},0\}\leq u_1+\cdots+u_{i}\leq \min\{\alpha_t,\beta_i\}-b_{i_1}$ for $i=k+1,\ldots,t$. Let $T=[1,t]\setminus \{i_1\}$ and  $C=B\cup \{i_1\}$. Note that $|\overline{C}|<|\overline{B}|$, where $\overline{C}=[k]\setminus C$. By induction hypothesis, $$\mathrm{depth}\  S([1,t]\setminus \overline{B})/I([1,t]\setminus \overline{B})=\mathrm{depth}\  S(T\setminus \overline{C})/I(T\setminus \overline{C})\neq 0, $$ that is, $P_{B\cup [k+1,t]}\notin \mathrm{Ass}(S/I).$ This completes the proof of our claim.

Now assume $\sum_{j\in \overline{B}}b_j <\min\{\alpha_t,\beta_{k+1}\}$. Then
$I(B\cup [k+1,t])$ is an ideal in $S(B\cup [k+1,t])$ generated by $\mathbf{x^u}$ with $\mathbf{u}$ satisfying $0\leq u_i\leq b_i$ for each $i\in B$ and
 $$\max\{\alpha_{i}-\sum_{j\in \overline{B}}b_j,0\}\leq \sum_{j\in B}u_j+u_{k+1}+\cdots+u_i\leq \min\{\alpha_t, \beta_{i}\}-\sum_{j\in \overline{B}}b_j$$ for $i=k+1,\ldots,t$.

 In view of Proposition~\ref{depthleft}, we see that $\depth\ S(B\cup [k+1,t])/I(B\cup [k+1,t])=0$ if and only if $\alpha_i<\min\{\alpha_t,\beta_i\}$ for $i=k+1,\ldots,t-1$.  This is the case if and only if $\alpha_{i}<\beta_{i}$  for $i=k+1,\ldots, t-1$ and  $\alpha_{t-1}<\alpha_t$, as required.
 \end{proof}

\begin{Proposition} \label{ass2} For any interval $[s, t]\subseteq [k+2,n]$, we have $P_{[s,t]}\in \mathrm{Ass}(S/I)$ if and only if $\alpha_i<\beta_i$ for $i=s,\cdots,t-1$, $\beta_{s-1}<\beta_s$, $\alpha_{t-1}<\alpha_t$ and $\beta_{s-1}<\alpha_t.$
\end{Proposition}

\begin{proof} Note that $I([1,n]\setminus \{k+1\})$ is an ideal in $S([1,n]\setminus \{k+1\})$ generated by $\xb^{\ub}$ with $\ub$ satisfying $u_i=0$ for $i=1,\ldots,k$,  and
\begin{eqnarray} \label{left}
\qquad \max\{\alpha_{i}-\beta_{k+1},0\}\leq u_{k+2}+\cdots+u_i\leq \beta_{i}-\beta_{k+1} \mbox{\quad for\quad} i=k+2,\ldots,n.
\end{eqnarray}
Hence $I([k+2,n])$ is an ideal in $S([k+2,n])$ generated by $\xb^\ub$ with $\ub$ satisfying the inequalities in (\ref{left}).
Now a direct application of  Proposition~\ref{lattice} yields this result.
\end{proof}

Combining the last three results,  we obtain immediately:

\begin{Theorem}  $P_A\in \mathrm{Ass}(S/I)$ if and only if either
\begin{enumerate}
\item[(a)] $A=B\cup [k+1,t]$, where $B\subseteq [1,k], t\geq k+1$ and  \[
\sum_{j\in [k]\setminus B}b_j <\min\{\alpha_t,\beta_{k+1}\},\quad\alpha_{t-1}<\alpha_t, \quad \alpha_{i}<\beta_{i}\text{ for all $i=k+1,\ldots, t-1$,}
\]
\end{enumerate}
or
\begin{enumerate}
\item[(b)] $A=[s,t]$, where $k+2\leq s\leq t\leq n$ and \[\beta_{s-1}<\beta_s,\quad \alpha_{t-1}<\alpha_t,\quad \beta_{s-1}<\alpha_t, \quad \alpha_i<\beta_i \text{ for all  $i=s,\ldots,t-1.$}
\]
\end{enumerate}
\label{howtowrite}
\end{Theorem}

\begin{Corollary}\label{astableft}  $\mathrm{astab}(I)=1$.
\end{Corollary}
\begin{proof}
The assertion follows from Theorem~\ref{howtowrite} together with Proposition~\ref{power}.
\end{proof}

Combining Corollary~\ref{astableft} with Corollary~\ref{dstableft} yields:

\begin{Corollary} Let $I$ be a left PLP-polymatroidal ideal. Then $\mathrm{dstab}(I)=\mathrm{astab}(I)=1$.
\end{Corollary}

It is proved in \cite{HV} that every polymatroidal ideal is of {\em strong intersection  type}, that is, if $I$ is a polymatroidal ideal,  then $\bigcap_{P\in \mathrm{Ass}(S/I)}P^{d_P}$ is an irredundant primary decomposition of $I$, where $d_P$ is the degree of generators in $I(P)$. From the proofs of Propositions~\ref{ass1} and \ref{ass2}, we see that if $P_{B\cup [k+1,t]}\in \mathrm{Ass}(S/I)$  then $d_{P_{B\cup [k+1,t]}}=\alpha_t-\sum_{j\in \overline{B}}b_j$, and if $P_{[s,t]}\in \mathrm{Ass}(S/I)$ then $d_{P_{[s,t]}}=\alpha_t-\beta_{s-1}$.

\begin{Example} \em
Let $I$ be an ideal generated by $\xb^{\ub}$ such that $\ub$ satisfies $0\leq u_i\leq 2$ for $i=1,2$, $0\leq u_3\leq 3$,
 $2\leq \sum_{i=1}^4u_i\leq 4$ and  $\sum_{i=1}^5u_i=5$.
 Then $I=P_{[1,5]}^5\cap P_{[1,4]}^2\cap P_{[1,5]\setminus \{3\}}^2\cap P_{[2,5]}^3\cap P_{[1,5]\setminus \{2\}}^3\cap P_{5}$. Here $P_5:=(x_5)$.
\end{Example}

\section{The depth and the associated prime ideals of right PLP-polymatroidal ideals}

In this section we will investigate another class of PLP-polymatroidal ideals, which we call {\em right PLP-polymatroidal ideals.} A PLP-polymatroidal ideal of type $(\mathbf{0},\bb|{\bm \alpha},{\bm \beta})$ is called a right PLP-polymatroidal ideal, if  there exists $k\in [1,n-1]$ such that   $b_i\geq d$ for all $i\leq k$, and $\alpha_i=\alpha_k$ and  $\beta_i=d$ for all $k+1\leq i\leq n-1$. Hence a right PLP-polymatroidal ideal  is an ideal $I$ in $S=K[x_1,\ldots,x_n]$ generated by the monomials $\xb^{\ub}$ with $\ub$ satisfying  $$\alpha_i\leq u_1+\cdots+u_i\leq \beta_i,  \mathrm{\ for\ } i=1,\ldots, k,$$
$$0\leq u_{i}\leq b_{i}, \mbox{\ for\ } i=k+1,\ldots,n,\qquad\quad$$ and
 $$u_1+\cdots+u_n=d.\qquad\qquad\qquad $$ Here $\alpha_1\leq \cdots  \leq \alpha_k\leq d,$ $0<\beta_1\leq \cdots  \leq \beta_k\leq d,$ $\alpha_i\leq \beta_i$ for $i=1,\ldots,k$ and $b_i>0$ for $i=k+1,\ldots,n$. To ensure that $I\neq 0$ we always assume that $d\leq \beta_k+b_{k+1}+\cdots+b_n$.  Note that if $k=1$ then $I$ is of Veronese type.

\medskip
In this section, the ideal $I$ always stands for the right PLP-polymatroidal ideal as given above and we will determine the depth and the associated prime ideals of $I$ and its powers.

\medskip
In order to illustrate the results of this section, we will use the following ideal as a running example.

\begin{Running Example} {\em  Consider the following inequalities:
\begin{center}

$\left\{\begin{array}{c}
  1\leq u_1\leq 3;\\
   3\leq u_1+u_2\leq 3;  \\
  3\leq u_1+u_2+u_3\leq 4;  \\
   6\leq u_1+\cdots+u_4\leq 6; \\
  7\leq u_1+\cdots+u_5\leq 8;\\
  8\leq u_1+\cdots+u_6\leq 9;
\end{array}\right.$\qquad
and \qquad
$\left\{\begin{array}{c}
 0\leq u_7\leq 2;\\
 \quad  0\leq u_8\leq 2;\\
   u_1+\cdots+u_8=12.\\
   \end{array}\right.$

\end{center}
Then all the vectors $\ub\in \ZZ_+^8$ satisfying all the inequalities above form the set of bases of a right PLP-polymatroid.  Let $I'$ denote its polymatroidal ideal. Thus,  $I'$ is a monomial ideal of the polynomial ring $S':=K[x_1,\ldots,x_8]$. }
\end{Running Example}

 We begin with a characterization of  the ideal $I$ for which  $\depth  S/I=0$.

\begin{Lemma} \label{right1} $\depth S/I=0$ if and only if $d\leq \beta_k+\mathbf{b}([k+1,n])-n+k$ and $\alpha_i<\beta_i$ for $i=1,\ldots,k$.
\end{Lemma}

\begin{proof} Suppose that $\depth S/I=0$. Then, by Proposition~\ref{nu1}, there is  ${\bold x}^{\ub}\in G(I)$ such that $N(\ub)=n-1$ (see Equation (\ref{nu}) for the definition of $N(\ub)$). This implies that   $u_1+\cdots+u_k\leq \beta_k-1$ and  $u_i\leq b_i-1$ for $i=k+1,\ldots,n-1$.  Therefore we have $d=u_1+\cdots+u_n\leq \beta_k-1+b_{k+1}-1+\cdots+b_{n-1}-1+b_n=\beta_k+\mathbf{b}([k+1,n])-n+k$.  Since $u_1+\cdots+u_i\leq \beta_i-1$, we have $\alpha_i<\beta_i$ for $i=1,\ldots,k$.

  Conversely, since $0<d-\beta_k+1\leq (b_{k+1}-1)+\cdots+(b_{n-1}-1)+b_n$, there exist integers $c_{k+1},\ldots,c_n$ such that $c_{k+1}+\cdots+c_n=d-\beta_k+1$, $0\leq c_i\leq b_i-1$ for $i=k+1,\ldots,n-1$ and $c_n\leq b_n$. Let $\ub=(\beta_1-1,\beta_2-\beta_1,\ldots,\beta_k-\beta_{k-1},c_{k+1},\ldots,c_n)$. Then $\xb^{\ub}\in G(I)$ and  $N(\ub)=n-1$, which implies $\depth S/I=0.$\end{proof}

\begin{Lemma}  \label{depthright1} Suppose that $\alpha_i<\beta_i$ for $i=1,\ldots,k$. Then $$\depth S/I=\max\{0, d- \beta_k-\mathbf{b}([k+1,n])+n-k\}.$$
\end{Lemma}

\begin{proof}  Let $t=d-\beta_k-\bb([k+1,n])+n-k$.   The case when  $t\leq 0$ follows from Lemma~\ref{right1}. Assume that $t>0$. Then we set $j=n-1-t$. Note that $j\geq k-1$, since $d\leq \beta_k+\mathbf{b}([k+1,n])$. There are several cases to consider.

\vspace{2mm}
If $j=k-1$, we let $\mathbf{u}=(\beta_1-1,\beta_2-\beta_1,\ldots,\beta_k-\beta_{k-1}+1, b_{k+1},\ldots, b_n).$  Then $\mathbf{x}^\mathbf{u}\in G(I)$ and $N(\mathbf{u})=k-1$. It follows that $\depth S/I\leq n-k=t$, see Proposition~\ref{nu1}.

\vspace{2mm}
 If $j=k$, we let   $\mathbf{u}=(\beta_1-1,\beta_2-\beta_1,\ldots,\beta_k-\beta_{k-1}, b_{k+1},\ldots,b_n).$ Then $\mathbf{x}^\mathbf{u}\in G(I)$ and $N(\mathbf{u})=k$. It follows that $\depth S/I\leq n-k-1=t$ by Proposition~\ref{nu1} again.

\vspace{2mm}
  If $j\geq k+1$, we let
  $\mathbf{u}=(\beta_1-1,\beta_2-\beta_1,\ldots,\beta_k-\beta_{k-1}, b_{k+1}-1,\ldots,b_j-1,b_{j+1},\ldots,b_n).$ Then $\mathbf{u}\in G(I)$ and $N(\mathbf{u})=j$. Hence $\depth S/I\leq n-1-j=t$.

\vspace{2mm}
  It remains to be shown that $N(\mathbf{u})\leq j$ if  $\mathbf{x}^\mathbf{u}\in G(I)$. Fix $\mathbf{u}$ with $\mathbf{x}^\mathbf{u}\in G(I)$. We set $$j_1=|\{i\:\; u_1+\cdots+u_i<\beta_i, 1\leq i\leq k\}|$$
  and $$j_2=|\{i\:\;u_i<b_i, k+1\leq i\leq n-1\}|.$$ Then $N(\mathbf{u})=j_1+j_2$.  If $j_1=k$, then $u_1+\cdots+u_k\leq \beta_k-1$ and so $u_{k+1}+\cdots+u_n\geq d-\beta_k+1$. On the other hand, $u_{k+1}+\cdots+u_n\leq \bb([k+1,n])-j_2$. This  implies that $d-\beta_k+1\leq \mathbf{b}([k+1,n])-j_2$, and so $j\leq n-k-t-1$. Therefore we have $N(\mathbf{u})=k+j_2\leq n-1-t=j$.

If $j_1<k$, then $d=u_1+\cdots+u_n\leq  \beta_k+ \bb([k+1,n])-j_2$,  and so $N(\mathbf{u})\leq j_2+k-1\leq n-1-t=j$. This completes the proof.
\end{proof}

\begin{Discussion}
 \label{product}
 \em We now consider the general case in which we do not require the strict inequalities $\alpha_i<\beta_i$ for $i=1,\ldots,k$. For this,  let $i_1<\cdots<i_s$ be such that $\{1\leq i\leq k\:\;\alpha_i=\beta_i\}=\{i_1,\ldots,i_s\}$. We then  observe that  $I$ can be write as the product of ideals $I_1,\ldots, I_s$ and  $J$ which are given as follows. The ideal  $I_1$ is generated by $x_1^{u_1}\cdots x_{i_1}^{u_{i_1}}$ with $u_1,\ldots,u_{i_1}$ satisfying  $\alpha_i\leq u_1+\cdots+u_{i}\leq \beta_{i}$ for $i=1,\ldots,i_1$. For $1\leq t\leq s-1$, the ideal $I_{t+1}$ is  generated by $$x_{i_t+1}^{u_{i_t+1}}\cdots x_{i_{(t+1)}}^{u_{i_{(t+1)}}}$$ with $u_{i_t+1},\ldots, u_{i_{(t+1)}}$ satisfying $$\alpha_{i}-\beta_{i_t}\leq u_{i_t+1}+\cdots+u_i\leq \beta_{i}-\beta_{i_t}$$ for $i=i_t+1,\ldots,i_{(t+1)}$.

If $i_s=k$, that is,  if $\alpha_k=\beta_k$,  then the ideal $J$ is generated by $x_{k+1}^{u_{k+1}}\cdots x_n^{u_n}$ with $u_{k+1},\ldots,u_n$ satisfying  $0\leq u_{k+1}\leq b_{k+1}, \ldots, 0\leq u_n\leq b_n, u_{k+1}+\cdots+u_n=d-\beta_k$. In this case $J$ is of Veronese type.

If $i_s<k$, that is, if $\alpha_k\neq \beta_k$, then the  ideal $J$ is generated by $x_{i_s+1}^{u_{i_s+1}}\cdots x_n^{u_n}$ with $u_{i_s+1},\ldots,u_n$ satisfying
$$\alpha_{i}-\beta_{i_s}\leq u_{i_s+1}+\cdots+u_i\leq \beta_{i}-\beta_{i_s}\mbox{\ for\ } i=i_s+1,\ldots,k,$$
$$ 0\leq u_{k+1}\leq b_{k+1}, \ldots, 0\leq u_n\leq b_n,$$
 and
$$u_{i_s+1}+\ldots+u_{k+1}+\cdots+u_n=d-\beta_{i_s}.$$
Note that  these ideals $I_1,\ldots, I_s, J$ have pairwise disjoint supports.

Moreover  if we let $S_t$ be the polynomial ring in variables $x_i$ with $i\in \mathrm{supp}(I_t)$ for $1\leq t\leq s$, then $\depth S_t/(I_t\cap S_t)=0$ by Proposition~\ref{simple}.

\end{Discussion}

Let $a\in\RR$. We denote by $\lfloor a \rfloor$ the lower integer part and by $\lceil a\rceil$ the upper integer part of $a$.

\begin{Theorem} \label{depth2} Let $s=|\{i\:\; \alpha_i=\beta_i,1\leq i\leq k\}|$. Then $$\depth S/I=s+\max\{0, d- \beta_k-\bb([k+1,n])+n-k-\lfloor\frac{\alpha_k}{\beta_k}\rfloor\}.$$
\end{Theorem}

\begin{proof} Let $R=K[x_{i_s+1},\ldots,x_n]$. By Discussion~\ref{product} and Lemma~\ref{H}(a), we have $\depth S/I=s+\depth R/R\cap J$.  If $\alpha_k=\beta_k$, then $\depth R/R\cap J=\max\{0, d- \beta_k-\bb([k+1,n])+n-k-1\}$  by Lemma~\ref{depthright1} (or by \cite[Corollary 4.7]{HRV}). If $\alpha_k<\beta_k$, then $\depth R/R\cap J=\max\{0, d- \beta_k-\bb([k+1,n])+n-k\}$  by Lemma~\ref{depthright1}. Combining two cases yields our formula.
\end{proof}

\begin{Corollary}\label{dstabright}
\begin{enumerate}

\item[(a)] If $d=\beta_k+\bb([k+1,n])$, then $\mathrm{dstab}(I)=1$.
\item[(b)] Set $\delta=\lfloor\frac{\alpha_k}{\beta_k}\rfloor$.  If $d<\beta_k+\bb([k+1,n])$, then $\mathrm{dstab}(I)=\lceil \frac{n-k-\delta}{\beta_k+\bb([k+1,n])-d} \rceil$.

 \end{enumerate}
\end{Corollary}
\begin{proof}
(a) If  $d=\beta_k+\bb([k+1,n])$, then $I=x_{k+1}^{b_{k+1}}\cdots x_n^{b_n}L$, where $L$ is generated by $x_1^{u_1}\cdots x_k^{u_k}$ such that $\alpha_i\leq u_1+\cdots+u_i\leq \beta_i$ for $i=1,\ldots,k-1$ and $u_1+\cdots+u_k=\beta_k$. Let $S_1=K[x_1,\ldots,x_k]$ and $S_2=K[x_{k+1},\ldots,x_n]$. By Lemme~\ref{H} and since $\depth S_2/(x_{k+1}^{b_{k+1}}\cdots x_n^{b_n})^m =n-k-1$ for all integer $m>0$, we have
 \[
 \depth S/I^m=\depth S_1/(S_1\cap L)^m+n-k.
 \]
 Note that $S_1\cap L$ is a LP-polymatroidal ideal. Therefore, by Proposition~\ref{lattice}, we have $\mathrm{dstab} (S_1\cap L)=1$, and so $\mathrm{dstab}(I)=1$.
\medskip

(b) If $d<\beta_k+\bb([k+1,n])$, then  $\depth S/I^m=s$ for all  $m\gg 0$, by Theorem~\ref{depth2}. It follows that $\mathrm{dstab}(I)$ is the smallest integer $i$ such that $i(d-\beta_k-\bb([k+1,n]))+n-k-\delta\leq 0$, which certainly  is $\lceil \frac{n-k-\delta}{\beta_k+\bb([k+1,n])-d} \rceil$.
\end{proof}

\begin{Running Example} {\em We continue the running example.

\noindent Let $I_1'$ be the ideal generated by monomials
$x_1^{u_1}x_2^{u_2}$ with $1\leq u_1\leq 3$ and $u_1+u_2=3$;

\noindent Let $I_2'$ be the ideal generated by monomials
$x_3^{u_3}x_4^{u_4}$ with $0\leq u_3\leq 1$ and $u_3+u_4=3$;

\noindent Let $J'$ be the ideal generated by monomials
$x_5^{u_5}\cdots x_8^{u_8}$ with $1\leq u_5\leq 2$ and $2\leq u_5+u_6\leq 3$; $0\leq u_7\leq 2$, $0\leq u_8\leq 2$, $u_5+\cdots+u_8=6$.

Then $I'=I'_1I'_2J'$ and $(I')^m=(I'_1)^m(I'_2)^m(J')^m$ for any $m>0$. Note that $(I'_i)^m$ is a LP-polymatroidal ideal for each $i=1,2$ and each $m>0$. Hence, by Lemma~\ref{H}(a) and Lemma~\ref{depthright1}, we have
$$\depth S'/(I')^m=2+\depth R'/(R'\cap J')=2+\max\{0,2-m\}.$$
 Here $R':=K[x_5,\ldots,x_8]$.  In particular, $$\mathrm{dstab}(I')=2.$$

 To decide the value of  $\mathrm{astab}(I')$, we only need to describe the associated prime ideals of $S'/J'$ and determine    $\mathrm{astab}(J')$, see Lemma~\ref{H}(b). This is what we do from Lemma~\ref{6.8} to Theorem~\ref{astabright}. }
\end{Running Example}

\begin{Lemma} \label{6.8}  Let $P_A\in \mathrm{Ass}(S/I)$. If there is $1\leq i\leq k$ such that $i\notin A$, then either $A\cap [1,i-1]=\emptyset$ or $A\cap [i+1,n]=\emptyset$.
\end{Lemma}

As an immediate consequence we obtain

\begin{Corollary} \label{asspossible} Let $P_A\in \mathrm{Ass}(S/I)$. Then  $A$ is an interval contained in $[1,k]$, or  $A$ is a subset of $[k+1,n]$, or  $A=[s,k]\cup B$ for some $s\leq k$ and some subset $B$ of $[k+1,n]$.
\end{Corollary}

\begin{Proposition} \label{assright} Set $\alpha_0=\beta_0=0$. Suppose that $\alpha_i<\beta_i$ for $i=1,\ldots,k$.   Given $1\leq s\leq t\leq  k$ and  a nonempty subset $B$ of $[k+1,n]$, we have
\begin{enumerate}

\item[(a)]The following statements are equivalent:
 \begin{enumerate}
\item[(1)] $P_{[s,t]}\in \mathrm{Ass}(S/I)$.
\vspace{1mm}

\item[(2)] $\alpha_k\leq d-\bb([k+1,n])\leq \beta_k$  and  $(\beta_{s-1}-\alpha_t)(\beta_{s-1}-\beta_s)(\alpha_{t-1}-\alpha_t)<0$.
\end{enumerate}

\vspace{2mm}
\item[(b)] The following statements are equivalent:
\begin{enumerate}
\item[(1)] $P_B\in \mathrm{Ass}(S/I)$.
\vspace{1mm}

\item[(2)] $\bb([k+1,n]\setminus B)+\beta_k <d\leq \beta_k+\bb([k+1,n])-|B|$.

\end{enumerate}

\vspace{2mm}

 \item[(c)] The following statements are equivalent:
\begin{enumerate}
  \item[(1)]
  $P_{[s,k]\cup B}\in \mathrm{Ass}(S/I)$.
  \vspace{1mm}

  \item[(2)] $\bb([k+1,n]\setminus B)+\max\{\beta_{s-1},\alpha_k\}< d\leq \beta_k+\bb([k+1,n])-|B|$ and $\beta_{s-1}<\beta_s$.
 \end{enumerate}

 \vspace{2mm}
\end{enumerate}
 \noindent Moreover, $\mathrm{Ass}(S/I)$ consists exactly  of these monomial prime ideals which satisfy one of the conditions described in {\em(a)}, {\em (b)} and \em{ (c)}.
\end{Proposition}

 \begin{proof}
(a) Note that  $I([1,k])$ is generated by monomials $\xb^{\ub}\in S([1,k])$ with $\ub$ satisfying  $\alpha_1\leq u_1\leq \beta_1, \ldots, \alpha_k\leq u_1+\cdots+u_k\leq \beta_k,$ and $u_1+\cdots+u_k=d-\bb([k+1,n])$. Hence $I([1,k])\neq 0$ if and only if $\alpha_k\leq d-\bb([k+1,n])\leq \beta_k$. If this is the case, then $I([1,k])$ is the LP-polymatroidal ideal generated by monomials $\xb^{\ub}\in S([1,k])$ with $\ub$ satisfying
\[
\alpha_i\leq u_1+\cdots+u_i\leq \min\{d-\bb([k+1,n]),\beta_i\}
\]
for $i=1,\ldots,k-1$, and
\[
u_1+\cdots+u_k=d-\bb([k+1,n]).
\]

Set $d_1=d-\bb([k+1,n]$ and apply Proposition~\ref{lattice} to this case, we see that $P_{[s,t]}\in \mathrm{Ass}(S/I)$ if and only if $\alpha_k\leq d_1\leq \beta_k$, $\alpha_{t-1}<\alpha_t$, $\min\{d_1, \beta_{s-1}\}<\alpha_t$, $\min\{d_1,\beta_{s-1}\}<\min\{d_1,\beta_{s}\}$ and $\alpha_i<\min\{d_1,\beta_i\}$ for $i=s,\ldots,t-1$, which is equivalent to $\alpha_k\leq d_1\leq \beta_k$, $\beta_{s-1}<\beta_s$, $\beta_{s-1}<\alpha_t$ and $\alpha_{t-1}<\alpha_t$, since $\alpha_t\leq \alpha_k$. Now our result follows.
\vspace{2mm}

(b) By using Proposition~\ref{basic}(b) repeatedly, we have   $I([k+1,n])$ is generated by monomials $\xb^\ub\in S([k+1,n])$  with $\ub$ satisfying $0\leq u_i\leq b_i$ for $i=k+1,\ldots,n$ and $u_{k+1}+\cdots+u_n=d-\beta_k$. In a similar way as in Proposition~\ref{ass1}, we see that if $d-\beta_k-\bb([k+1,n]\setminus B)\leq 0$ then $P_B\notin \mathrm{Ass}(S/I)$, moreover if
$d-\beta_k-\bb([k+1,n]\setminus B)> 0$, then
$I(B)$ is generated by monomials $\xb^\ub\in S(B)$  with $\ub$ satisfying
\[ 0\leq u_i\leq b_i, \text{ for all $i\in B$}\quad  \text{and}\quad  \sum_{i\in B} u_i=d-\beta_k-\bb([k+1,n]\setminus B).
\]
Applying Lemma~\ref{right1} and Proposition~\ref{simple}(a), the result follows.

\vspace{2mm}
(c) Note that $I([s,k]\cup B)$ is generated by monomials $\xb^\ub\in S([s,k]\cup B)$ with $\ub$ satisfying
\[
\max\{\alpha_i-\beta_{s-1},0\}\leq u_s+\cdots+u_i\leq \beta_i-\beta_{s-1} \quad  \text{for $i=s,\ldots,k$}
\]
\[
0\leq u_i\leq b_i \; \text{for $i\in B$} \quad \text{and}\quad u_s+\cdots+u_k+\sum_{i\in B}u_i=d-\beta_{s-1}-\bb([k+1,n]\setminus B).
\]
 Hence $I([s,k]\cup B)\neq (0)$ if and only if $d-\beta_{s-1}-\bb([k+1,n]\setminus B)>\max\{0, \alpha_k-\beta_{s-1}\}$. If this is the case, we let $\hat{d}=d-\beta_{s-1}-\bb([k+1,n]\setminus B).$  We  describe $I([s,k]\cup B)$ in the form as given in the beginning of this section: $I([s,k]\cup B)$ is generated by monomials $\xb^\ub\in S([s,k]\cup B)$ with $\ub$ satisfying
\[
\max\{\alpha_i-\beta_{s-1},0\}\leq u_s+\cdots+u_i\leq \min\{\hat{d}, \beta_i-\beta_{s-1}\} \quad  \text{for $i=s,\ldots,k$}
\]
\[
0\leq u_i\leq b_i \; \text{for $i\in B$} \quad \text{and}\quad u_s+\cdots+u_k+\sum_{i\in B}u_i=\hat{d}.
\]
Note that $\max\{\alpha_i-\beta_{s-1},0\}< \min\{\hat{d}, \beta_i-\beta_{s-1}\}$ for $i=s,\ldots,k$ if and only if $\beta_{s-1}<\beta_s$. By Lemma~\ref{right1} and Proposition~\ref{simple}(a), we have $P_{[s,k]\cup B}\in \mathrm{Ass}(S/I)$ if and only if $\beta_{s-1}<\beta_s$ and $\hat{d}\leq \min\{\hat{d}, \beta_k-\beta_{s-1}\}+\bb(B)-|B|$ . One can check the last inequality is equivalent to $d\leq \beta_k+\bb([k+1,n])-|B|,$ as required.

\medskip
Finally the last statement follows from Corollary~\ref{asspossible}.
\end{proof}

Following \cite{HRV}, we use  $\mathrm{Ass}^{\infty}(S/I)$ to denote the set  of prime  ideals $P$ for which $P\in \mathrm{Ass}(S/I^m)$ for all $m\gg 0$. Since any polymatroidal ideal has the strong persistence property by \cite[Proposition 2.4]{HQ},
we see that $$\mathrm{Ass}^{\infty}(S/I)=\bigcup_{m=1}^{\infty}\mathrm{Ass}(S/I^m).$$

\begin{Proposition} \label{assright2} Suppose $\alpha_i<\beta_i$ for $i=1,\ldots,k$ and $d< \beta_k+\bb([k+1,n])$. Given $1\leq s\leq t\leq k$ and  a subset $B$  of $[k+1,n]$, we have
\begin{enumerate}
\item[(a)] $P_{[s,t]}\in \mathrm{Ass}^{\infty}(S/I)$ $\Longleftrightarrow$ $P_{[s,t]}\in \mathrm{Ass}(S/I)$.

\vspace{2mm}
\item[(b)] $P_B\in \mathrm{Ass}^{\infty}(S/I)$ $\Longleftrightarrow$  $\bb([k+1,n]\setminus B)+\beta_k <d$.
\vspace{2mm}

\item[(c)] $P_{[s,k]\cup B}\in \mathrm{Ass}^{\infty}(S/I)$ $\Longleftrightarrow$ $\bb([k+1,n]\setminus B)+\max\{\beta_{s-1},\alpha_k\}< d$ and $\beta_{s-1}<\beta_s.$

\end{enumerate}
\vspace{2mm}

\noindent Moreover, $\mathrm{Ass}^{\infty}(S/I)$ consists exactly of these monomial  prime ideals  which satisfy one of the conditions described in {\em(a)}, {\em(b)} and {\em (c)}.
\end{Proposition}

\begin{proof} By the assumption that  $d< \beta_k+\bb([k+1,n])$, we have  $md\leq m\beta_k+m\bb([k+1,n])-|B|$ for all  $m\gg 0$. By this fact and in view of Propositions~\ref{assright} and~\ref{power},   the assertions (a), (b) and (c) follow. Finally the last statement follows from the last sentence of Proposition~\ref{assright}.
\end{proof}

\begin{Running Example} \em{ We continue the running example. By Proposition~\ref{assright}, we obtain:
$$\mathrm{Ass}(S'/J')=\{P_{[5,6]},P_5,P_6 \}\cup \{P_7,P_8\}\cup \{P_{[5,7]},P_{\{5,6,8\}},P_{[6,7]},P_{\{6,8\}}\};$$
and
$$\mathrm{Ass}^{\infty}(S'/J')=\mathrm{Ass}(S'/(J')^2)=\mathrm{Ass}(S'/J')\cup\{P_{[3,4]}\}\cup\{P_{[1,4]},P_{[2,4]}\}.$$
}
Here $P_i:=P_{\{i\}}=(x_i)$. In particular, $$\mathrm{astab}(I')=\mathrm{astab}(J')=2.$$
\end{Running Example}
\begin{Corollary} \label{special}
  Suppose that $\alpha_i<\beta_i$ for $i=1,\cdots,k$ and $d< \beta_k+\bb([k+1,n])$. Then
\[\mathrm{astab}(I)=\lceil \frac{n-k}{\beta_k+\bb([k+1,n])-d} \rceil.
\]
\end{Corollary}

\begin{proof}
Set $\Delta=\lceil \frac{n-k}{\beta_k+\bb([k+1,n])-d} \rceil$. In view of Proposition~\ref{assright2}(c), we have $P_{[n]}=P_{[k]\cup [k+1,n]}\in \mathrm{Ass}^{\infty}(S/I)$, since $\beta_0=0<\beta_1$. Moreover, $P_{[n]}\in \mathrm{Ass}(S/I^m)$ if and only if $m\geq \Delta$  by  Proposition~\ref{assright}(c) and Proposition~\ref{power}. This implies $\mathrm{astab}(I)\geq \Delta$.

It remains to be shown that $P\in \mathrm{Ass}(S/I^m)$ for any $m\geq \Delta$ and for any $P\in \mathrm{Ass}^{\infty}(S/I)$.
 But this is clear by checking it for three classes of ideals  in $\mathrm{Ass}^{\infty}(S/I)$ given in Proposition~\ref{assright2} respectively.
\end{proof}

We remark that if $k=1$, then Corollary~\ref{special} is equivalent to  \cite[Corollary 4.6]{HRV}.
In the proof of the following result we will use the following fact: if $I$ is a monomial ideal in $S=K[x_1,\ldots,x_n]$, then $\mathrm{Ass}(T/IT)=\{PT\:\; P\in \mathrm{Ass}(S/I)\}$ for any polynomial ring $T=K[x_1,\ldots,x_m]$ with $m\geq n$. As before we set $\delta=\lfloor \frac{\alpha_k}{\beta_k}\rfloor$.

\begin{Theorem} \label{astabright} \begin{enumerate}
\item[(a)] If $d=\beta_k+\bb([k+1,n])$, then $\mathrm{astab}(I)=1$.

\item[(b)] If $d<\beta_k+\bb([k+1,n])$, then  $\mathrm{astab}(I)=\lceil \frac{n-k-\delta}{\beta_k+\bb([k+1,n])-d} \rceil$.

\end{enumerate}
\end{Theorem}

\begin{proof} (a) As in the proof  of Theorem~\ref{depth2}, $I=x_{k+1}^{b_{k+1}}\cdots x_n^{b_n}J$, where $J$ is a lattice path polymatroidal ideal. By Proposition~\ref{lattice}, we have $\mathrm{astab}(J)=1$ and so $\mathrm{astab}(I)=1$.

(b) Let $I_1,\ldots,I_s,J$ and $S_1,\ldots,S_s, R$ be as in Discussion~\ref{product}. Since $S_i\cap I_i$ is a lattice path polymatroidal ideal, we have $\mathrm{astab}(S_i\cap I_i)=1$ for  $i=1,\ldots,s$. It follows that $\mathrm{astab}(I)=\mathrm{astab}(J\cap R)$ by Lemma~\ref{H}(b). There are two cases to consider.

In the case when $\alpha_k=\beta_k$ we have  $\mathrm{astab}(R\cap J)=\lceil \frac{n-k-1}{\beta_k+\bb([k+1,n])-d} \rceil$ by \cite[Corollary 4.6]{HRV},
and in the case when $\alpha_k<\beta_k$ we have  $\mathrm{astab}(R\cap J)=\lceil \frac{n-k}{\beta_k+\bb([k+1,n])-d}\rceil$ by Corollary~\ref{special}.
This yields the desired formula.
\end{proof}

We conclude this  section and this paper by the following result.

\begin{Corollary} Let $I$ be a right PLP-polymatroidal ideal. Then $\astab(I)=\dstab(I)$.

\end{Corollary}

\begin{proof} It follows from  Theorem~\ref{astabright} as well as Corollary~\ref{dstabright}.
\end{proof}

{\bf \noindent Acknowledgement:} Thank the referee very much for his/her many suggestions which have improved the presentation of this paper a lot.

\end{document}